\documentclass[12pt,reqno]{amsart}

\usepackage{amsmath}
\usepackage{amsfonts}
\usepackage{amscd}
\usepackage{amssymb}
\usepackage{amsthm}
\usepackage{mathdots}
\usepackage{mathtools}
\usepackage{multicol}
\usepackage{inputenc}
\usepackage{caption}
\usepackage{subcaption}
\usepackage{float}
\usepackage{commath}

\numberwithin{equation}{section}

\usepackage{MnSymbol}
\usepackage[linktocpage=true,colorlinks,backref, citecolor=green!50!black]{hyperref}
\usepackage{url}
\usepackage{mathrsfs}

\usepackage{tikz,tikz-cd, color}
\usepackage{adjustbox}

\usepackage{appendix}
\usetikzlibrary{matrix}
\usetikzlibrary{decorations.pathmorphing}

\tikzset{
  symbol/.style={
    draw=none,
    every to/.append style={
      edge node={node [sloped, allow upside down, auto=false]{$#1$}}}
  }
}

\usepackage{stmaryrd}
\usepackage{enumerate}
\usepackage{xcolor}
\usepackage[all]{xy}
\usepackage{fullpage}

\usepackage{pgfplots}
\pgfplotsset{compat=1.18}
\usepgfplotslibrary{fillbetween}

\newtheorem{thm}{Theorem}[section]
\newtheorem{lem}[thm]{Lemma}
\newtheorem{prop}[thm]{Proposition}
\newtheorem{cor}[thm]{Corollary}

\theoremstyle{definition}
\newtheorem{defn}[thm]{Definition}
\newtheorem{exmp}[thm]{Example}

\newtheorem{conv}[thm]{Convention}

\theoremstyle{remark}
\newtheorem{rem}[thm]{Remark}

\newcommand{\oo}{\mathcal{O}}
\newcommand{\bP}{\mathbb{P}}
\newcommand{\bC}{\mathbb{C}}
\newcommand{\bZ}{\mathbb{Z}}
\newcommand{\bR}{\mathbb{R}}

\newcommand{\bU}{\mathbb{U}}
\newcommand{\cT}{\mathcal{T}}
\newcommand{\cF}{\mathcal{F}}

\newcommand{\cH}{\mathcal{H}}
\newcommand{\Hom}{{\rm Hom}}
\newcommand{\lHom}{\mathcal{H}om}
\newcommand{\Ext}{{\rm Ext}}
\newcommand{\lExt}{\mathcal{E}xt}
\newcommand{\ext}{{\rm ext}}
\newcommand{\Quot}{{\rm Quot}}

\newcommand{\ch}{{\rm ch}}
\newcommand{\im}{{\rm im\,}}

\newcommand{\stable}{{\rm s}}
\DeclareMathOperator{\length}{length}

\DeclareMathOperator{\supp}{supp}
\DeclareMathOperator{\sing}{sing}

\newcommand{\Geo}{\mathrm{Geo}}

\newcommand{\DLP}{\mathcal{C}_{\mathrm{DLP}}}

\DeclareMathOperator{\chbar}{\overline{ch}}
\DeclareMathOperator{\Coh}{Coh}
\DeclareMathOperator{\Db}{D^b}
\DeclareMathOperator{\GL}{GL}

\DeclareMathOperator{\gr}{gr}

\title{Extensions and Segre stratifications\\ over algebraic surfaces}
\author[]{Thomas Goller}
\address{Temple University, PA, USA}
\email{thomas.goller@temple.edu}
\author[]{Yinbang Lin}
\address{Department of Mathematics, University of Houston, Houston, TX, USA} 
\email{yinbang.lin@gmail.com}

\keywords{Stable sheaves, algebraic surfaces, Segre stratifications, Lange's conjecture, Brill--Noether problems}
\subjclass[2020]{Primary 14D20, 14F06; Secondary 14J60, 14F08}

\begin{document}
\begin{abstract}
    We study two closely related topics over algebraic surfaces: stability of extensions of stable sheaves by stable sheaves, and maximal subsheaves of a given stable sheaf. For the former, we provide a construction of complete families of sheaves via extensions and prove the stability for some cases. The construction enables us to prove certain cases of Weak Brill--Noether over rational surfaces. The second topic leads to a refinement of the Segre stratification of the moduli of sheaves. We obtain the expected dimension of certain refined Segre strata over rational surfaces and K3 surfaces.
    Over the projective plane, we prove our main result that all refined Segre strata corresponding to line bundles are irreducible of the expected dimension and are nested under taking closures. As a consequence, we obtain more cases of the stability of extensions. Our results suggest that refined Segre strata behave much better than Brill--Noether strata. Our study of the refined Segre strata relies crucially on Bridgeland stability conditions and the work of Li and Zhao.
\end{abstract}
\maketitle

\section{Introduction}

In the study of coherent sheaves on a smooth projective variety $X$, a key construction is the short exact sequence or extension
\begin{equation}\label{eq:ext}
    0 \to F \to E \to G \to 0,
\end{equation}
where $E$, $F$, and $G$ are coherent sheaves on $X$. Two main perspectives on \eqref{eq:ext} are as follows:
\begin{itemize}
\item Fixing $F$ and $G$, what sheaves $E$ arise as extensions?
\item Fixing $E$, what subsheaves $F$ and quotient sheaves $G$ arise?
\end{itemize}
These equivalent perspectives are directly related to many important problems.
Focusing mainly on the case when $X$ is a surface, we explore various aspects of such short exact sequences, obtaining results related to Brill--Noether theory, Lange's conjecture, and Segre stratifications, as well as to the study of Quot schemes and Bridgeland wall-crossing.
We are motivated indirectly by the Verlinde/Segre correspondence \cite{GotMel22,Marian25} and Quot scheme approaches to studying Le Potier's strange duality conjecture \cite{MarOpr07,goller2019rankone,GolLin25}.

We work over the field $\mathbb{C}$ of complex numbers. Table \ref{tab:notation} summarizes the notational conventions used throughout the paper.

\begin{table}[h]
    \centering
    \begin{tabular}{c  l}
        $X$ & smooth projective surface \\
        $H$ & ample divisor on $X$ \\
        $E$ & coherent sheaf on $X$ \\
       $e$ & class $[E]$ of $E$ in the Grothendieck group $K(X)$ \\
       $M(e)$ & moduli space of semistable sheaves $E$ of class $e$ \\
       $E \in M(e)$ & $E$ is a semistable sheaf of class $e$ \\
       $r$, $c_i$, $\ch_i$ & rank, $i$th Chern class, $i$th Chern character\\
        $\mu$ & slope $(c_1\cdot H)/r$\\
        $\Delta$ & 
        $(c_1^2-2r\ch_2)/(2r^2)$\\
        $\chi$ & holomorphic Euler characteristic\\
        $\hom$, $\ext^i$ & dimension of the corresponding $\Hom$ or $\Ext$ group\\
        $I_{Z_m}$ & ideal sheaf of a zero-dimensional subscheme of length $m$

    \end{tabular}
     \caption{Notation}
    \label{tab:notation}
\end{table}

Let $F$ and $G$ be stable sheaves and $E$ be a general non-split extension \eqref{eq:ext}. A natural question is whether $E$ is stable. Over curves, this question
is attributed to Lange \cite{Lange83Klass} and
has been resolved \cite{RussoTei99,ballico2000extensions}. 
We study this question over surfaces.
In the case when the slopes of $F$ and $G$ are as close as possible, we obtain an affirmative answer to the question (Proposition \ref{prop:lange-small-slope-diff-coprime}). A next case that is especially tractable is when the extension is expected to produce a general stable bundle for dimension reasons. To study such extensions, we introduce a construction of a complete family (Lemma \ref{lem:complete}). This construction enables us to prove over rational surfaces certain cases of Weak Brill--Noether (Proposition \ref{prop:bn}) in the sense of \cite{CosHui18weakBN} and stability of extensions (Proposition \ref{prop:lange-hirzebruch}). 
Our methods also yield a shorter proof of certain cases of Lange's conjecture
(Corollary \ref{cor:lange-cur}).

We shift our perspective from extensions to subsheaves. 
In the study of Lange's conjecture over curves \cite{Lange83Klass,RussoTei99}, the authors studied the stratification of the moduli space via {\em Segre invariant} \eqref{eq:segre}.
We extend the study to surfaces.
As the resulting strata are not in general irreducible, we refine the stratification by further considering the discriminant
(Definition \ref{defn:segre-refined}).
We call a subsheaf with the corresponding Segre invariant and the minimum discriminant a {\em Segre subsheaf} (Definition \ref{defn:segre-sub}).
A Segre subsheaf is maximal among subsheaves of $E$ of a fixed rank.

Turning our attention to $\bP^2$, we particularly focus on the refined stratum of sheaves admitting $\oo(k)$ as a Segre subsheaf, which we call the $\oo(k)$-{\em Segre stratum} (Definition \ref{defn:F-Segre-stratum}).
In this case, the weak Brill--Noether result of G\"ottsche and Hirschowitz \cite{GotHir98} is enough to deduce the {\em general} Segre invariant of rank 1, the one corresponding to the open stratum. We complete the picture with the following main result describing all Segre strata in which the Segre subsheaf is locally free of rank 1. For a class $e\in K(X)$, let $k_e$ be the largest integer $k < \mu(e)$ such that $\chi(e(-k))>0$.

\begin{thm}\label{thm:O(k)-Segre-strata} Let $X=\bP^2$ and let $e$ be a class of rank $\ge 1$ such that $M(e)$ is of positive dimension. 
Let $k$ be an integer satisfying $k_e \le k < \mu(e)$. Then the $\oo(k)$-Segre stratum in $M(e)$ is non-empty and irreducible,
and its closure
contains the $\oo(j)$-Segre stratum for all $k\le j<\mu(e)$.
If $k=k_e$, the stratum is open and dense in $M(e)$, while if $k > k_e$, the stratum has codimension $1-\chi(e(-k))$ in $M(e)$.
Except in the case where $k=k_e$ and
$M(e-[\oo(k_e)])$ is empty, the general sheaf in the $\oo(k)$-Segre stratum is stable and arises as a general extension
\begin{align}\label{eq:ext-small-strata}
    0 \to \oo(k) \to E \to G \to 0,
\end{align}
where $G$ is a general stable sheaf.
\end{thm}

We prove the theorem in three main stages.
\begin{enumerate}[Step 1:]
    \item We prove that the $\oo(k)$-Segre stratum in $M(e)$ is non-empty and contains a unique component of maximum dimension, whose general member is a general extension \eqref{eq:ext-small-strata} as described. For $k=k_e$, we do this using a general construction of a complete family (Lemma \ref{lem:complete}).
    For $k>k_e$, we use the machinery of Bridgeland wall crossings and careful dimension estimates.
    \item We use the result of Step 1 and elementary modifications to obtain a similar result for the $I_{Z_m}(k)$-Segre strata (Proposition \ref{prop:rk-1-subsh-p2}).
    \item We combine the results of Step 1 and Step 2 with a degeneracy locus argument (Proposition \ref{prop:degeneracy-locus-codim}) to prove that the $\oo(k)$-Segre strata are irreducible and nested.
\end{enumerate}

Step 1 is the most difficult part of the proof. While it is clear that every sheaf $E$ in the $\oo(k)$-Segre stratum admits a short exact sequence \eqref{eq:ext-small-strata}, controlling the dimension of such $E$ requires control over $G$, for instance showing that $G$ is stable.
Alternately, one can start with stable $G$ and consider such extensions, but then one must show $E$ is stable, as in Lange's conjecture.
The case of rank 2 vector bundles $E$ was studied by Roa-Leguizam\'on, Torres-L\'opez, and Zamora \cite{RL-TL-Zam21}, where stability can be checked via the vanishing of $H^2(E)$.
Our main strategy to handle arbitrary ranks
is to use Bridgeland wall crossings to control the loci of $E$ for which $G$ is not general, building on the beautiful work by Li and Zhao \cite{LiZhao19}.

Combining Proposition \ref{prop:rk-1-subsh-p2} with the theorem yields the following main corollary about stability of extensions over $\bP^2$.

\begin{cor}\label{cor:lange-p2} Over $X=\mathbb{P}^2$, let $g$ be a class such that $M(g)$ is of positive dimension and $k \in \bZ$ satisfy $k<\mu(g)$.
Let $G \in M(g)$ be a general stable sheaf. Then, a general non-split extension $E$ as in \eqref{eq:ext} is stable in the following cases:
    \begin{enumerate}[(a)] 
        \item $F\cong\oo(k)$, assuming $\chi(g(-k-1))\le 0$ and $M(g+[\oo(k)])$ is of positive dimension;
            \item $F\cong I_{Z_m}(k)$ for general $Z_m \in (\bP^2)^{[m]}$, assuming $r(g)\geq 1$ and $\chi(g(-k))\le 0$.
        \end{enumerate}
\end{cor}

We also use Theorem \ref{thm:O(k)-Segre-strata} to deduce a number of other results on $\bP^2$.
\begin{itemize}
    \item We deduce a lower bound on the dimension of certain Brill--Noether strata (Corollary \ref{cor:BN}), which in many cases shows that the Brill--Noether strata have a component of dimension larger than the expected dimension (Remark \ref{rem:BN-expdim}).
    \item We show that every general stable sheaf admits a two-step resolution by line bundles (Corollary \ref{cor:gen-res}), extending Gaeta's resolution for ideal sheaves to higher ranks. We also show that for $k>k_e$, general stable sheaves in the $\oo(k)$-Segre stratum admit a two-step resolution involving four line bundles (Corollary \ref{cor:small-strata-res}).
    \item We derive a bound on general Segre invariants of all ranks (Corollary \ref{cor:gen-Segre-bound}).

\end{itemize}

We organize the paper as follows.
In \S \ref{sec:ext}, we focus on extensions, including a construction of a complete family of sheaves and some consequences related to stability of extensions and Weak Brill--Noether.
In \S \ref{sec:segre}, we define the Segre invariant and the refined Segre stratification. We also obtain the expected dimension of certain refined Segre strata over rational surfaces and K3 surfaces. 
The remainder of the paper is focused on developing our main results on the projective plane.
In \S \ref{sec:p2-general}, we use classical approaches to prove the case $k=k_e$ of Theorem \ref{thm:O(k)-Segre-strata} and make careful Euler characteristic calculations related to the Dr\'ezet--Le Potier curve. In \S \ref{sec:bridgland}, we provide background on Bridgeland stability conditions and use wall-crossing calculations to complete Step 1 of the proof of Theorem \ref{thm:O(k)-Segre-strata}. In \S \ref{sec:conseq}, we complete Steps 2 and 3 of the proof of Theorem \ref{thm:O(k)-Segre-strata}, prove Corollary \ref{cor:lange-p2}, and deduce results related to Brill--Noether strata and resolutions of stable sheaves by line bundles. 

\vskip8pt
\textit{Acknowledgments.} We thank Xiaolei Zhao and Zhixian Zhu for helpful discussions. Claude was used to help produce \textsc{Tikz} figures and for copy-editing assistance.  

\section{Extensions over general surfaces}\label{sec:ext}

This section contains results related to complete families of sheaves, as well as special cases in which stability of extensions can be proven.

\subsection{Complete families via extensions}

Over a smooth projective curve or surface, for classes $f$ and $g$, let $\xi=\min \{\,\ext^1(G,F)\mid (F,G)\in M^\stable(f)\times M^\stable(g)\,\}$, which is assumed to be positive.
Define
\begin{equation}\label{eq:non-special-locus}
    N=\{\,(F,G)\in M^{\mathrm{s}}(f)\times M^{\mathrm{s}}(g)\mid
        \ext^1(G,F)=\xi
        ,\ \Ext^1(F,G)=0
    \,\}.
\end{equation}

\begin{lem}\label{lem:complete}
    Let $P$ denote the following family of extensions:
    \begin{align}\label{eq:ext-complete}
        P=\{\,\text{nontrivial }0\to F\to E\to G\to 0\mid (F,G)\in N\,\}/\bC^*.
    \end{align}
    Then $P$ parametrizes a complete family of sheaves $E$.
\end{lem}
Notice that $P$ is a projective bundle over $N$.

\begin{proof}
    Let $\alpha$ denote an extension of the form
    $0\to F\to E\to G\to 0$ and let $[\alpha]\in P$ denote the corresponding class in the projectivized extension bundle. 
Let $A$ be a mini-versal deformation space of $E$, which is given by a complete local noetherian algebra, with a mini-versal family $\mathcal{E}_A$ over $A\times X$, which is given by a pro-family. See \cite[\S 2.9 \& Thm. 3.6]{Nitsure-deformation}.
Let $0\in A$ denote the unique closed point.  

Let $\pi_A\colon A\times X \to A$ denote the projection and $Q=\Quot_{\pi_A}(\mathcal{E}_A,g)$ denote Grothendieck's relative Quot scheme. 
Let $[q]\in Q$ denote the equivalence class of the quotient $q$. 
Then, we have the following exact sequence
\begin{align*}
    0\to \Hom(F, G)\to T_{[q]}Q\xrightarrow{
    } T_0 A\cong \Ext^1(E,E)\xrightarrow{\omega_+}\Ext^1(F,G). 
\end{align*}
Since $\Ext^1(F,G)=0$ by assumption,
the Kodaira--Spencer map
$T_{[q]}Q\to T_0 A\cong \Ext^1(E,E)$ is surjective. 
Given a quotient $q^\prime \colon E^\prime \to G^\prime$, we can complete it to a short exact sequence $0\to F^\prime \to E^\prime \to G^\prime\to 0$. 
Due to the openness of stability, if $[q^\prime]$ is close to $[q]$, we can assume that $F^\prime$ and $G^\prime$ are also stable.
Given two representatives of $[q^\prime]$, say $q_1^\prime$ and $q^\prime_2$, there is a nonzero scalar $a$ such that $q^\prime_1=aq^\prime_2$. 
Therefore, the corresponding extensions are also equivalent. 
Thus, locally around $[q]$ and $[\alpha]$, there is a natural map from $Q$ to $P$.
Its tangential map makes the following diagram commutative:
\begin{equation*}
    \begin{tikzcd}[column sep=small]
T_{[q]}Q \arrow[rr] \arrow[dr, two heads] &   & T_{[\alpha]}P \arrow[dl] \\
                                   & T_0A. &                  
\end{tikzcd}
\end{equation*}
Here, the slanted arrows are Kodaira--Spencer maps. 
Therefore, the Kodaira--Spencer map $T_{[\alpha]}P\to T_0A$ is also surjective. Thus, the family $P$ of sheaves $E$ is complete.
\end{proof}

\begin{rem}
    The naive map sending an extension to the quotient does not provide a well-defined map from $P$ to $Q$, because equivalent extensions might produce non-equivalent quotients. However, if the middle term $E$ is simple, the map is well-defined. 
\end{rem}

\begin{rem}\label{rem:F=L^r}
Suppose that $L$ is a line bundle and $L^\ell$ is the only semistable sheaf of class $[L^\ell]$. Then in the definition of $N$, 
we can allow $F\cong L^\ell$ and Lemma \ref{lem:complete} still holds.
The proof needs to be modified as follows.
We consider the frame bundle $I$ over $Q$, where the data contain $[q]$ and an isomorphism $g\colon F\to L^\ell$. Then, we have a commutative diagram of linear maps:
    \begin{equation*}
        \begin{tikzcd}
            T_{([q],g)} I \arrow{r} \arrow[two heads]{d} & T_{[\alpha]}P \arrow{d}\\
            T_{[q]} Q \arrow[two heads]{r} & T_0 A.
        \end{tikzcd}
    \end{equation*}
    Then, the right arrow is also surjective and the family $P$ is complete.
\end{rem}

When the difference of slopes is not too small compared to the genus, the lemma yields another proof of Lange's conjecture over curves, based on the weak Brill--Noether result and approximations by stable bundles.
\begin{cor}\label{cor:lange-cur}
    Over a smooth projective curve $C$ of genus $\geq 2$, let $F$ and $G$ be two general stable vector bundles of classes $f$ and $g$ respectively, such that $\mu(F)<\mu(G)$ and $\chi(F,G)\geq 0$. Then, a general nontrivial extension $E$ of $G$ by $F$ is stable.  
\end{cor}

\begin{proof}
    According to a theorem of Hirschowitz, see e.g. \cite[Thm. 1.2]{RussoTei99}, $\Ext^1(F,G)=0$. On the other hand, $\Hom(G,F)=0$ by stability. Therefore, $N$ as defined in \eqref{eq:non-special-locus} is non-empty. Then, by Lemma \ref{lem:complete}, $P$ defines a complete family of vector bundles $E$. 
    By \cite[Prop. 2.6]{NarasimhanRamanan-deformation}, every vector bundle deforms to a stable bundle. 
    Therefore, a general member of this complete family is stable. 
\end{proof}

\subsection{Weak Brill--Noether}\label{sec:wbn}
    If 
    \begin{equation}\label{eq:bir-rul}
        \pi\colon X \to C
    \end{equation} is a birationally ruled surface, then
    a coherent sheaf $E$ on $X$ is \emph{prioritary} (with respect to $\pi$) if it is torsion-free and if $\Ext^2(E, E(-\alpha_p)) = 0$ for all $  p \in C$, where $\alpha_p = \pi^{-1}(p)$ denotes the fiber over $p$. These $\alpha_p$ are all numerically equivalent, and we write $\alpha \in \operatorname{NS}(X)$ for the numerical class of these $\alpha_p$.

For a polarization $H$ satisfying the condition 
\begin{equation}\label{eq:pol-pri}
    H\cdot (K_X+\alpha)<0,
\end{equation}
every $H$-semistable sheaf $E$ is prioritary \cite[Pf. of Thm. 1]{Wal98}. The stack of prioritary sheaves of fixed Chern classes is irreducible \cite[Prop. 2]{Wal98}. So is the moduli scheme \cite[Thm. 1]{Wal98}.

For rational surfaces, the general construction of complete families of sheaves (Lemma \ref{lem:complete}) allows us to show the following.

\begin{prop}\label{prop:bn}
    Over a rational surface $X\not=\bP^2$, fix a birational ruling \eqref{eq:bir-rul} and an ample divisor $H$ satisfying \eqref{eq:pol-pri}.
    Let $e$ be a class of rank $r>0$, first Chern class $c_1$, and second Chern class $c_2$, where $c_1$ is the class of a smooth subcurve. 
    Under the condition
    \begin{align*}
        \frac{c_1^2+c_1\cdot K_X}{2}+\alpha\cdot c_1\leq c_2\leq\frac{c_1^2-c_1\cdot K_X}{2},
    \end{align*}
    a general sheaf in $M(e)$
    has vanishing higher cohomology.
\end{prop}
Namely, in these situations, a general sheaf satisfies the weak Brill--Noether property.

\begin{proof}
We take $G$ to be the push-forward of a general line bundle $L$ on a smooth subcurve representing $c_1$ satisfying $c_2(G)=c_2$.
Consider extensions of the form
\begin{align}\label{eq:ext-by-triv}
    0\to \oo^r\to E\to G\to 0.
\end{align} 
Notice that $\Hom(G,\oo^r)=0$ and $\Ext^2(G,\oo^r)\cong \Hom(\oo^r,G(K_X))^\vee$. 
The conditions on $c_2$ ensure $\chi(G(K_X))\leq 0$, $\chi(G(K_X+\alpha))\leq 0$, and $\chi(G)\geq 0$. 
According to weak Brill--Noether for curves \cite[Prop. 3.8]{Coskunetal-survey24}, 
$\Hom(\oo,G(K_X))=\Hom(\oo,G(K_X+\alpha))=\Ext^1(\oo,G)=0$. In particular, since $G$ has vanishing higher cohomology, so does $E$. 

We show that $E$ is prioritary.
By Serre duality, we have $\Ext^2(E,E(-\alpha))\cong \Hom(E,E(\alpha+K_X))^\vee$. We also have the exact sequence
\begin{align*}
    0\to \Hom(G,E(\alpha+K_X)) \to \Hom(E,E(\alpha+K_X)) \to \Hom(\oo^r,E(\alpha+K_X)).
\end{align*}
Since $G$ is torsion, $\Hom(G,E(\alpha+K_X))=0$. 
On the other hand, we have the exact sequence
\begin{align*}
    0\to H^0(\oo^r(\alpha+K_X))\to H^0(E(\alpha+K_X)) \to H^0(G(\alpha+K_X)).
\end{align*}
Since $H^0(\oo(\alpha+K_X))=0$ and $H^0(G(K_X+\alpha))=0$,
$H^0(E(\alpha+K_X))=0$.
Therefore, $E$ is prioritary. 

Taking $F \cong \oo^r$ in the definition of $N$,
the family $P$ of extensions as in \eqref{eq:ext-complete} is a projective bundle over $M^\stable(g)$. 
According to Lemma \ref{lem:complete}, the family $P$ is complete. Thus, $P$ parametrizes a complete family of prioritary sheaves. Hence, a general extension $E$ is stable and it is a general member in $M(e)$. 
\end{proof}
The case over $\mathbb{P}^2$ is fully understood \cite{GotHir98}. The construction \eqref{eq:ext-by-triv}, together with elementary modifications, was applied in \cite{Gould-Lee-Liu22} over $\mathbb{P}^2$ to study special sheaves. 

\subsection{Special cases of stability of extensions}

Let $d=c_1\cdot H$ denote the degree. 
\begin{prop}\label{prop:lange-small-slope-diff-coprime}
    Over a smooth projective surface $X$, suppose $H$ is a rational ample divisor such that $\{\,c_1(E)\cdot H\mid E\mbox{ is a coherent sheaf}\,\}=\bZ$. Let $F$ and $G$ be general stable sheaves. Let $E$ be a non-split extension of $G$ by $F$. Assume that
    \begin{align}\label{eq:farey}
        r(F)d(G)-r(G)d(F)=1,
    \end{align}
    that is, the reduced fractions $\mu(F)$ and $\mu(G)$ form a Farey pair. 
    Then $E$ is stable in the following two situations:
    \begin{enumerate}[(a)]
        \item $r(F)\geq 2$;
        \item $r(F)=1$, $\chi(G,F)+r(G)c_2(F)<0$, and the extension $E$ is general.
    \end{enumerate}
    \end{prop} 
For a coherent sheaf $E$, we denote its rank and degree as $v(E)=(r(E),d(E))$.    
    By Pick's theorem, the condition \eqref{eq:farey} is equivalent to the following:
        \begin{align}\label{eq:small-triangle}
        \text{There is no integral point in the following closed parallelogram} \nonumber\\
        \text{formed by the vectors $v(F)$ and $v(G)$ other than the vertices.}
    \end{align} 

\begin{proof}
    Suppose $F^\prime\subset E$ is a stable saturated subsheaf that destabilizes $E$. In particular, $\mu(F^\prime)\geq \mu(E)>\mu(F)$. Furthermore, according to \eqref{eq:small-triangle}, $\mu(F^\prime)>\mu(E)$. Then, the composition $F^\prime\hookrightarrow E\twoheadrightarrow G$, which we denote by $\phi$, is non-zero. Therefore, $\mu(F^\prime)\leq\mu(G)$.  
    We claim that $v(F^\prime)=v(G)$.
    If not, $v(\im \phi)=v(G)$
    according to \eqref{eq:small-triangle} and $\ker \phi\not=0$. Then, $\ker \phi$ is contained in $F$. Thus, the line segment connecting $v(G)$ and $v(F^\prime)$ is not above the one connecting $v(G)$ and $v(E)$. Therefore, the integral point $v(F^\prime)$ lies on the parallelogram,
    contradicting \eqref{eq:small-triangle}.
    Thus, $v(F^\prime)=v(G)$ and $\phi$ is injective. Moreover, $T:= \operatorname{coker} \phi$ has dimension 0.

    If $T=0$, then the extension $E$ splits, which is a contradiction. So, we assume $T\not=0$. 
    Notice that $v(F)=v(G^\prime)$ 
    and $\ch_2(F)-\ch_2(G^\prime)=\ch_2(F^\prime)-\ch_2(G)=-\ch_2(T)<0$.
    We form the following commutative diagram of exact sequences:
    \begin{equation}\label{eq:special-lange-proof}
        \begin{tikzcd}
        & & 0\arrow{d} & 0\arrow{d} &\\
            & & F^\prime \arrow[equal]{r} \arrow{d} & F^\prime \arrow{d} &\\
            0 \arrow{r} & F \arrow[equal]{d} \arrow{r} &E \arrow{r} \arrow{d} & G\arrow{r} \arrow{d} &0\\
            0 \arrow{r} & F \arrow{r} & G^\prime\arrow{r} \arrow{d} & T\arrow{r} \arrow{d} &0\\
            & & 0 &0. &
        \end{tikzcd}
    \end{equation}

    If $\supp T\not\subset \sing F$, then $G^\prime$ contains non-trivial torsion, which contradicts the assumption that $F^\prime$ is saturated.
    Since in the case (a), a general $F$ is locally free, we have proven this case.

    In case (b), since $r(F)=1$, a general $F$ is the ideal sheaf of a reduced scheme, say $Z$, of dimension 0, twisted by a line bundle, say $L$. 
    Since $G^\prime$ is torsion free, $W:=\supp(T) \subset Z$. Thus, we have $G' \cong L \otimes I_{Z \setminus W}$, which in particular is stable. 
    Because $F$ and $G$ are general, we can assume that $\sing(G) \cap \sing(F) = \emptyset$.
    Since $\hom(G,G')=0$ due to stability and $\Ext^1(G,T)=0$, applying $\Hom(G,-)$ to the short exact sequence $0 \to F \to G' \to T \to 0$ yields a short exact sequence
    \begin{align}\label{eq:special-lange-proof-ses}
         0 \to \Hom(G,T) \to \Ext^1(G,F) \to \Ext^1(G,G') \to 0.
    \end{align}
    The diagram \eqref{eq:special-lange-proof} implies that the extension $E$ comes from a map $G\to T$ under the connecting homomorphism in \eqref{eq:special-lange-proof-ses}. 
    Since there are finitely many choices of $T$, to prove the statement, it suffices to show the first map is not surjective. This is guaranteed by the condition $\chi(G,F)+r(G)c_2(F)<0$, because $\ext^1(G,F)\geq
    -\chi(G,F)$ and $\length(T)\leq c_2(F)$.  
\end{proof}
The proposition generalizes part of \cite[Lem. 2.1]{Yoshioka99}, without the condition on the Picard number.
Stability of extensions of the tangent sheaf by the structure sheaf has been proven over log Fano/Calabi--Yau manifolds using the K\"ahler--Einstein metric \cite{Tian92Fano,Li2021}.
Some cases over the smooth quadric have been stated for the anticanonical polarization in \cite[Thm. 6.4]{raha2025higherrankcliffordstheorem}.

\begin{exmp}
This example is to illustrate that the condition $\chi(G,F)+r(G)c_2(F)<0$ in Proposition \ref{prop:lange-small-slope-diff-coprime} (b) is reasonable. 
Suppose that $\chi(G,F)+r(G)c_2(F)\geq 0$. Furthermore, we assume that $K_X\cdot H<0$. Other conditions remain the same.
Let $F_0=F^{\ast\ast}$. Then we have the following commutative diagram of exact sequences:
\begin{equation*}
    \begin{tikzcd}
        0\ar[r] & F \ar[r] \ar[d,hook] & E\ar[r] \ar[d, hook] & G\ar[r] \ar[d,equal] &0\\
         0\ar[r] & F_0 \ar[r] & E_0\ar[r] & G\ar[r] &0.
    \end{tikzcd}
\end{equation*}
Moreover, $\chi(G,F_0)=\chi(G,F)+\chi(G, F_0/F)=\chi(G,F)+r(G)c_2(F)\geq 0$.
Notice that $\hom(G,F_0)=0$ 
    by stability and $\ext^2(G,F_0)=\hom(F_0,G\otimes K_X)=0$, by stability and \eqref{eq:small-triangle}.  
    We expect $\ext^1(G,F_0)=0$, and thus $E_0\cong F_0\oplus G$. 
    Hence, we should have a nonzero map $E\hookrightarrow E_0\twoheadrightarrow F_0$. Thus, $E$ should not be semistable. 
\end{exmp}

Next, we apply our construction of complete families to the Hirzebruch surfaces $X=\mathbb{F}_n$, where $n\geq 0$. 
Let $\pi\colon \mathbb{F}_n\to \mathbb{P}^1$ be the corresponding projection, $\alpha$ be the fiber class, and $\beta$ be the class of a section such that $\beta^2=-n$. 
As in \S \ref{sec:wbn}, we fix an ample divisor $H$ such that $H\cdot(K_X+\alpha)<0$. Let $\nu(g)=c_1(g)/r(g)$ denote the total slope of a class $g$.

\begin{prop}\label{prop:lange-hirzebruch}
    Suppose $g\in K(X)$ such that $r(g)\geq 1$, $\chi(g)\geq 0$, $\mu(g)>0$, and 
    \begin{itemize}
        \item $\alpha\cdot \nu(g)=-1$; or
        \item $-1<\alpha\cdot \nu(g)\leq 1$ and $\beta\cdot \nu(g)\geq-1$. 
    \end{itemize}
    Assume $G\in M^\stable(g)$ is a general sheaf and that there is a stable sheaf of class $\ell[\oo]+g$, where $\ell$ is a positive integer. Then for a general nontrivial extension 
    \begin{equation*}
        0\to \oo^\ell\to E\to G\to 0, 
    \end{equation*}
    $E$ is stable.
\end{prop}

\begin{proof}
    According to \cite[Cor. 3.7 (1) \& Cor. 3.10]{CosHui20}, $G$ has vanishing higher cohomology.
    Recall from Remark \ref{rem:F=L^r} that we can similarly define $N$ as in \eqref{eq:non-special-locus} by taking $F=\oo^\ell$. Since we assume that there are nontrivial extensions, $N$ is non-empty. Then $P$ as in \eqref{eq:ext-complete} provides a complete family. 

    Since $\nu(g)\cdot \alpha\leq 1$, $\nu(G(\alpha+K_X))\cdot \alpha\leq -1$. Then $\Ext^1(G,\oo^\ell(-\alpha))\cong \Hom(\oo^\ell,G(\alpha+K_X))^\vee=0$ according to \cite[Cor. 3.7 (2)]{CosHui20}. 
    We can show that $E$ is prioritary. Since it's similar to the corresponding argument in the proof of Proposition \ref{prop:bn}, we omit the proof here. 
    Since we have a complete family of prioritary sheaves, a general member is stable.
\end{proof}

\begin{rem}
    Note that the hypotheses of Proposition \ref{prop:lange-hirzebruch} force $\ext^1(G,\oo)\geq \ell$: otherwise, a general extension of $G$ by $\oo^\ell$ would contain a summand $\oo$ and hence be unstable. See \cite[p.16]{BerGolJoh23}.
\end{rem}

\section{Segre stratifications}\label{sec:segre}
We review the Segre invariant for a coherent sheaf, which was first introduced over curves, and the corresponding stratification of the moduli space. 
We then refine the stratification and demonstrate an example to justify the necessity for the refinement.
This section includes some core definitions.

\subsection{Segre invariants and stratifications}

We fix a smooth projective surface $X$ with an ample divisor $H$.
Let $E$ be a coherent sheaf and $r^\prime$ be an integer. When $0<r^\prime <r(E)$, 
    the $r^\prime$-th {\em Segre invariant} of
    $E$
    is 
    \begin{equation}\label{eq:segre}
        S_{r^\prime}(E)=\min_{F\subset E \mbox{ subsheaf of rank } r^\prime}\{\,(r^\prime c_1(E)-r(E)c_1(F))\cdot H\,\}. 
    \end{equation}
As a function, $S_{r'}$ is lower semicontinuous \cite[Lem. 2.6]{RL-TL-Zam21}. 
If $E$ is semistable, then $S_{r^\prime}(E)\geq 0$.
A variant of the Segre invariant was used by Langer \cite{Langer06} to provide an effective result on the irreducibility of the moduli space of sheaves on surfaces.

    For a semistable sheaf $E$, let $\gr^{\mathrm{JH}} (E)$ denote the associated grading of one of its Jordan--H\"older filtrations. For a fixed $e \in K(X)$,
    let 
    \begin{equation*}
        Z(r^\prime,s)=\{\,E\in M(e)\mid S_{r^\prime}(\gr^{\mathrm{JH}}(E))=s\,\}.
    \end{equation*}
We have a decomposition
\begin{equation}\label{eq:segre-strat}
    M(e)=\coprod_s Z(r^\prime,s).
\end{equation}
The Segre invariant of rank $r^\prime$ achieved by a general element $E\in M(e)$, namely, the Segre invariant corresponding to the open stratum in \eqref{eq:segre-strat}, is called the {\em general Segre invariant} of rank $r^\prime$.

\begin{lem}\label{lem:bdd}
    Fixing $r^\prime$ and $e$, the set $\{\,S_{r^\prime}(E)\mid E\in M(e)\,\}$ is finite.
    The family 
\begin{equation}\label{eq:fam-sub-segre}
\left\{\, F \;\middle|\;
\begin{array}{l}
E \in M(e),\ F \subset E \text{ is a subsheaf of rank } r^\prime,\\[0.3em]
(r' c_1(E) - r(E)c_1(F)) \cdot H = S_{r'}(E), \text{ and}\\[0.3em]
E/F \text{ is torsion free}
\end{array}
\,\right\}
\end{equation}
is bounded.
\end{lem}
\begin{proof}There is a positive integer $n$ such that every $E\in M(e)$ can be realized as a quotient $\oo(-n)\otimes \bC^{\chi(E(n))}\twoheadrightarrow E$. Choosing $r'$ general elements of $\bC^{\chi(E(n))}$ yields
an injective map $\oo(-n)^{r^\prime}\to E$.
Then, $S_{r^\prime}(E)\leq (r^\prime c_1(E)+nr(E)H)\cdot H$ for all $E$, proving the first statement.
Notice that the quotients $E/F$ appearing in \eqref{eq:fam-sub-segre} are also quotients of $\oo(-n)^{\oplus \chi(E(n))}$ and have slopes bounded from above. Thus, the family of such quotients is bounded, according to a theorem of Grothendieck's (see e.g. \cite[Lem. 1.7.9]{HuyLeh10}). So the family \eqref{eq:fam-sub-segre} is bounded as well. 
\end{proof}

Recall that $\Delta(E)$ denotes the discriminant of $E$.
According to Lemma \ref{lem:bdd}, there are only finitely many possible classes of the subsheaves $F$. We can make the following definition.

\begin{defn}\label{defn:segre-refined}
Let
    \begin{align*}
        T_{r^\prime}(E)=\min\left\{\,\Delta(F)\,\middle|\,
        \begin{array}{l}
            F\subset E,\ r(F)=r^\prime,  \\
            (r^\prime c_1(E)-r(E)c_1(F))\cdot H=S_{r^\prime}(E)  
        \end{array}\,\right\}.
    \end{align*}This is the minimum discriminant among all rank-$r^\prime$ subsheaves of $E$ realizing the Segre invariant $S_{r^\prime}(E)$. Let 
    \begin{align*}
        Z(r^\prime, s,t)=\{\,E \in M(e)\mid S_{r^\prime}(\gr^{\mathrm{JH}}(E))=s,\ T_{r^\prime}(\gr^{\mathrm{JH}}(E))=t\,\}. 
    \end{align*}
\end{defn}

Note that $S_{r'}(E)$ and $T_{r'}(E)$ are both invariant under tensoring $E$ by a line bundle.

\begin{defn}\label{defn:segre-sub} We say that $F$ is a \emph{Segre subsheaf} of $E$ if there is an inclusion $F \hookrightarrow E$ and $F$ satisfies $(r(F)c_1(E)-r(E)c_1(F)) \cdot H = S_{r(F)}(E)$ and $\Delta(F) = T_{r(F)}(E)$. If the inclusion is clear from context, we may say that $F \subset E$ is a Segre subsheaf.
\end{defn}

The following is clear.
\begin{lem}
If $F \subset E$ is a Segre subsheaf, then $F$ is maximal among subsheaves of rank $r(F)$.
\end{lem}

The converse of the lemma does not hold in general. For example, consider $E\cong I_p \oplus \oo$. The summand $I_p$, which is the ideal sheaf of a point, is maximal, but not a Segre subsheaf. 

It will be helpful to note the following observation.
\begin{lem}\label{lem:segre-sub-lf}
    If $E$ is locally free, any maximal subsheaf of fixed rank is locally free. In particular, its Segre subsheaves are locally free.
\end{lem}

For each $0<r^\prime <r(E)$, 
\begin{align*}
    M(e)=\coprod_{s,t} Z(r^\prime,s,t). 
\end{align*}
Next, we equip $Z(r^\prime,s)$ and $Z(r^\prime,s,t)$ with
scheme structures.

Let $Q=\Quot(\oo(-m)^N,e)$ be the Quot scheme in a GIT construction of $M(e)$, such that $M(e)=\Quot^{\rm ss}(\oo(-m)^N,e)/\mkern-3mu/\GL(N,\mathbb{C})$. 
Let $\mathcal{E}\to Q^{\rm ss}\times X$ be the universal quotient, $\pi\colon Q^{\rm ss}\times X\to Q^{\rm ss}$ the projection, and $\varpi\colon Q^{\rm ss}\to M(e)$ the modular map. 
Considering Lemma \ref{lem:bdd}, let $s_0=\min \{\,S_{r^\prime}(E)\mid E\in M(e)\,\}$ and $t_0=\min \{\,T_{r^\prime}(E)\mid E\in M(e),\ S_{r^\prime}(E)=s_0\,\}$. We consider the relative Quot scheme 
$\Quot_{\pi}(r^\prime,s_0,t_0; \mathcal{E})$
parametrizing quotients whose kernels have rank $r^\prime$, first Chern class determined by $s_0$, and discriminant $t_0$.
Note that the image of the relative Quot scheme in $Q^{\rm ss}$ is stable under the $\GL(N,\mathbb{C})$-action.
Then, we can equip the subset $Z(r^\prime,s_0,t_0)$ with the structure of the closed subscheme $\varpi\circ\pi(\Quot_{\pi}(r^\prime,s_0,t_0; \mathcal{E}))$. Note that if $E$ admits a rank $r^\prime$ subsheaf giving the invariants $s_0$ and $t_0$, then its S-equivalence class lies in $Z(r^\prime,s_0,t_0)$.

We next consider $M_1=M(e)\setminus Z(r^\prime,s_0,t_0)$ and the pullback family $\mathcal{E}|_{\varpi^{-1}(M_1)\times X}$. Let $s_1=\min \{\,S_{r^\prime}(E)\mid E\in M_1\,\}$ and $t_1=\min \{\,T_{r^\prime}(E)\mid E\in M_1,\ S_{r^\prime}(E)=s_1\,\}$; note that it is possible that $s_1=s_0$.
Using the relative Quot scheme $\Quot_{\pi}(r^\prime,s_1,t_1; \mathcal{E}|_{\varpi^{-1}(M_1)\times X})$,
we can similarly equip $Z(r^\prime,s_1,t_1)\subset M_1$ with the structure of a closed subscheme. 
We can carry out this construction inductively, which will terminate after finitely many steps, according to Lemma \ref{lem:bdd}.
All the quotients at each step are torsion-free.

On the other hand, notice that
\begin{align}
    Z(r^\prime,s)=\coprod_t Z(r^\prime,s,t),
\end{align}
which is a finite union.

Thus, we have proven the following statement.
\begin{lem}\label{lem:segre-strata-loc-closed}
    Given $r^\prime$ and $M(e)$, the subsets $Z(r^\prime, s)$ and $Z(r^\prime,s,t)$ are locally closed.
\end{lem}

Since we will focus on rational surfaces and K3 surfaces, we recall the following important definitions: A bundle $F$ over a rational surface or K3 surface
is
\begin{itemize}
    \item \emph{rigid} if $\Ext^1(F,F)=0$;
    \item \emph{exceptional} if it is rigid, $\Hom(F,F)\cong \mathbb{C}$, and $\Ext^2(F,F)=0$;
    \item \emph{spherical} over a K3 surface if it is rigid and $\Hom(F,F)\cong\Ext^2(F,F)\cong \bC$.
\end{itemize} 
The following notation will be convenient for referring to specific refined Segre strata without specifying the values of the invariants.

\begin{defn}\label{defn:F-Segre-stratum}
For $F$ a rigid bundle, the refined Segre stratum in $M(e)$ of sheaves that admit $F$ as a Segre subsheaf is called the \emph{$F$-Segre stratum}. Similarly, for $D$ a divisor and $n$ a positive integer, the \emph{$I_{Z_n}(D)$-Segre stratum} in $M(e)$ is the refined Segre stratum where the Segre subsheaf is of the form $I_{Z_n}(D)$. 
\end{defn}

\begin{exmp}\label{ex:Segre-strata}
    Over $X=\bP^2$, consider the moduli space of sheaves with rank 2, first Chern class 1, and second Chern class $4$, which has dimension 12.
    The $\oo$-Segre stratum $Z(1,1,0)$ consists of sheaves that fit in the following  non-split short exact sequence
    \begin{align*}
    0\to\oo\to E \to I_{Z_4}(1)\to 0,
\end{align*}
where $Z_i\in X^{[i]}$, and has dimension 11. The $I_{Z_1}$-Segre stratum $Z(1,1,1)$ consists of sheaves $E$ that fit in the following non-split short exact sequence 
\begin{align*}
    0\to I_{Z_1}\to E \to I_{Z_3}(1)\to 0,
\end{align*}
and also has dimension 11.
The strata $Z(1,1,0)$ and $Z(1,1,1)$ are irreducible.
Moreover, we know the following:
\begin{itemize}
    \item $\overline{Z(1,1,1)}\cap Z(1,1,0)\not= \emptyset$; 
    \item $Z(1,1,0)\not\subsetneq \overline{Z(1,1,1)}$ and $Z(1,1,1)\not\subsetneq \overline{Z(1,1,0)}=Z(1,1,0)$;
    \item The stratum $Z(1,1)=Z(1,1,0)\coprod Z(1,1,1)$ is reducible.
\end{itemize}
\end{exmp}

\subsection{Expected dimensions of Segre strata}
We embed some Segre strata into degeneracy loci to understand the expected dimensions of the Segre strata.

Let $X$ be a rational surface or a K3 surface
and $H$ be an ample divisor. 
Let $F$ be a class $f$ bundle that is either exceptional over a rational surface or spherical over a K3 surface.
Let $\mathcal{E}$ be a flat family of semistable sheaves of class $e$ on $X$ parametrized by a finite-type scheme $S$. 
Given a point $s\in S$, let $E_s$ denote the fiber $\mathcal{E}|_{\{s\}\times X}$. 
Denote the projection $S\times X \to S$ as $\pi$ and the other projection as $\pi_X$. 
Take a resolution of $F$ of the form 
\begin{equation*}
    0\to K\to V=\oo(-m)^\lambda \to F\to 0, 
\end{equation*}
such that $H^i(E_s(m))=0$ for all $s\in S$ and all $i>0$.
Suppose 
\begin{equation}\label{eq:van-ext2}
    \mu(f)<\min\{\,\mu(e), \ \mu(e)-K_X \cdot H\,\}.
\end{equation}
Then $\Ext^2(F,E_s)\cong \Hom(E_s, F\otimes K_X)^\vee=0$ for all $s\in S$. 
Let $\lExt^i_{\pi}$ denote the $i$-th derived functor of $\lHom_{\pi}=\pi_* \circ\lHom$. 
Then we have a long exact sequence over $S$, which degenerates to the following form:
\begin{equation*}
    0\to \lHom_\pi(\pi_X^*F,\mathcal{E}) \to \lHom_\pi(\pi_X^*V,\mathcal{E}) \xrightarrow{\Phi} \lHom_\pi(\pi_X^*K,\mathcal{E}) \to \lExt_\pi^1(\pi_X^*F,\mathcal{E}) \to 0. 
\end{equation*}
Here, the middle two terms are locally free of ranks $\lambda \cdot \chi(e(m))$ and $\lambda \cdot \chi(e(m))-\chi(f, e)$, respectively.
For $j<\min\{\,\lambda \cdot \chi(e(m)),\ \lambda \cdot \chi(e(m))-\chi(f, e)\,\}$, define the degeneracy locus
\begin{equation*}
    D_j(\Phi)=\{\,s\in S\mid\Phi|_{\{s\}\times X}\colon \Hom(V,E_s)\to \Hom(K,E_s)\text{ has rank }\leq j\,\}. 
\end{equation*}
We define the $F$-Segre stratum in $S$ as the subset of points $s$ such that $E_s$ admits $F$ as a Segre subsheaf. This is similar to Definition \ref{defn:F-Segre-stratum}.
Then the following statement is clear, see e.g. \cite[p.242]{Ful98}.

\begin{prop}\label{prop:degeneracy-locus-codim}
Let $X$, $H$, $F$, and $S$ be as above. Assume $S$ is irreducible.
Suppose \eqref{eq:van-ext2} holds and $\chi(f,e)\leq 0$. 
Let $j_0=\lambda \cdot \chi(e(m))-1$. 
Then
\begin{equation*}
    D_{j_0}(\Phi)=\{\,s\in S \mid \text{$\exists$ a nonzero map } F\to E_s \,\},
\end{equation*}
it contains the $F$-Segre stratum in $S$, and every component has codimension at most $1-\chi(f,e)$.
\end{prop}

\section{Rank-one subsheaves over $\bP^2$: classical aspects}\label{sec:p2-general}

We study rank-one subsheaves of semistable sheaves over the projective plane. This is the most accessible situation, due to the complete understanding of the existence of stable sheaves \cite{DreLeP85} and the cohomology of general sheaves \cite{GotHir98}.

\subsection{Reduced classes and the Dr\'ezet--Le Potier curve}

Let $K(\bP^2)$ denote the Grothendieck group of $\bP^2$.
We take $[\ch_0,\ch_1,\ch_2]$ as homogeneous coordinates on the projective plane $\bP(K(\bP^2) \otimes \bR)$, viewing the locus given by $\ch_0=0$ as the line at infinity. The complement of this locus is an affine plane, which we denote $\bU$, with coordinates given by the {\em reduced Chern characters}
\begin{align*}
    \chbar_i=\frac{\ch_i}{\ch_0}
    \mbox{ for }i=1,2.
\end{align*}
Note that $\ch_0=r$, $\chbar_1=\mu$, and
the discriminant satisfies
    $\Delta=\frac{1}{2}\chbar_1^2-\chbar_2$.
For a class $e \in K(\bP^2)$, we refer to the corresponding point in $\bU$ as a \emph{reduced class}.
For distinct reduced classes $e$ and $f$, we write $L_{ef}$ for the line through $e$ and $f$ and $l_{ef}$ for the line segment from $e$ to $f$. These lines still make sense if $e$ or $f$ lie on the line at infinity.

\begin{conv}\label{conv:non-reduced}
We will frequently specify points in $\bU$ or lines $L_{ef}$
using non-reduced classes $e$ and $f$ or even sheaves $E$ and $F$.
\end{conv}

For any real number $a$, we let $\Delta_a$ denote the set
\[
    \Delta_a = \left\{ P \in \bU \mid \Delta(P)=a \right\}.
\]
According to the classical Bogomolov inequality \cite{Bogomolov78VectorBundles}, the reduced classes of slope-semistable
torsion-free sheaves lie on or below $\Delta_0$. We now explain the precise description of the classes for which $M(e)$ is non-empty.

The classes of the rank-one exceptional sheaves $\oo(k)$ for $k \in \mathbb{Z}$ lie on $\Delta_0$. All other exceptional classes lie strictly between $\Delta_0$ and $\Delta_{1/2}$. Given an exceptional class $\alpha$, consider three associated points in $\bU$ defined as follows. Let $\alpha^+ = \alpha-(0,0,1/\ch_0(\alpha))$. Let $\alpha^l$ denote the point in the intersection of  $L_{\alpha(-3)^+ \alpha^+}$ and $\Delta_{1/2}$ that has greater $\chbar_1$. Let $\alpha^r$ denote the point in the intersection of $L_{\alpha^+ \alpha(3)^+}$ and $\Delta_{1/2}$ that has smaller $\chbar_1$. Note that $L_{\alpha(-3)^+ \alpha^+}$ and $L_{\alpha^+ \alpha(3)^+}$ are the lines given by the linear equations $\chi(\alpha,-)=0$ and $\chi(-,\alpha)=0$.
There is a fractal curve $\DLP$, called the \emph{Dr\'{e}zet--Le Potier curve}, lying between $\Delta_{1/2}$ and $\Delta_1$ and containing the line segments $l_{\alpha^l \alpha^+}$ and $l_{\alpha^+ \alpha^r}$ for all exceptional characters $\alpha$. See Figure \ref{fig:s,q}, which shows the segments corresponding to exceptional bundles of rank 1 and 2. The following theorem highlights the key role played by this curve.

\begin{thm}[Dr\'{e}zet--Le Potier \text{\cite{DreLeP85}}]\label{thm:lpcurve} Given $e \in K(\bP^2)$ with $\ch_0(e)>0$, there exists a semistable coherent sheaf with class $e$ if and only if $e$ is proportional to an exceptional class or $e$ is on or below $\DLP$. In the latter case, $M(e)$ is irreducible of dimension $1-\chi(e,e)$, which is positive, and the general sheaf in $M(e)$ is slope-stable.
\end{thm}
For the dimension, see \cite[Prop. 6.9]{Mar78}. For the statement on slope-stability, see \cite[Thm. 4.11]{DreLeP85}.

\begin{figure}[h]
\centering
\begin{tikzpicture}
\begin{axis}[
    xmin=-3.6, xmax=3.6,
    ymin=-2, ymax=6,
    axis lines=middle,
    grid=both,
    axis equal,
    xtick={-2,0,2},
    ytick={0,2,4},
    major grid style={line width=0.3pt, draw=gray!20},
    axis line style={draw=gray!80, line width=0.6pt, ->},
    xlabel={$\chbar_1$},
    ylabel={$\chbar_2$},
    width=9cm, height=9cm,
    tick label style={font=\footnotesize, fill=white},
    every axis x label/.style={at={(current axis.right of origin)}, anchor=west},
    every axis y label/.style={at={(current axis.above origin)}, anchor=south}
]

    \addplot[domain=-4:4, samples=100, color=gray!60, line width=0.6pt] {x^2/2} node[pos=.15, anchor=west, color=gray] {\footnotesize $\Delta_0$};
    \addplot[domain=-4:4, samples=100, color=gray!60, line width=0.6pt] {x^2/2-1/2} node[pos=.8, anchor=west, color=gray] {\footnotesize $\Delta_{1/2}$};

    \addplot[only marks, mark=*, mark options={fill=green!50!black, draw=none}, mark size=0.8pt] coordinates {
        (-4, 8) (-3, 4.5) (-2, 2) (-1, 0.5) (0, 0) (1, 0.5) (2, 2) (3, 4.5) (4, 8)
        (-3.5,5.75) (-2.5,2.75) (-1.5,0.75) (-0.5,-0.25) (0.5,-0.25) (1.5,0.75) (2.5,2.75) (3.5,5.75)
    };

    \addplot[color=purple!80!blue, line width=0.8pt, no marks] coordinates {

        (-3.382, 5.218) (-3, 3.5) (-2.618, 2.927)

        (-2.382, 2.337) (-2, 1) (-1.618, 0.809) 

        (-1.382, 0.455) (-1, -0.5) (-0.618, -0.309) 

        (-0.382, -0.427) (0, -1) (0.382, -0.427) 

        (0.618, -0.309) (1, -0.5) (1.382, 0.455)

        (1.618, 0.809) (2, 1) (2.382, 2.337)

        (2.618, 2.927) (3, 3.5) (3.382, 5.218)

        (-3.586, 5.929) (-3.5, 5.5) (-3.414, 5.328)

        (-2.586, 2.843) (-2.5, 2.5)  (-2.414, 2.414)

        (-1.586, 0.757) (-1.5, 0.5) (-1.414, 0.5)

        (-0.586, -0.328) (-0.5, -0.5) (-0.414, -0.414)

        (0.414, -0.414) (0.5, -0.5) (0.586, -0.328)

        (1.414, 0.5) (1.5, 0.5) (1.586, 0.757)

        (2.414, 2.414) (2.5, 2.5) (2.586, 2.843)

        (3.414, 5.328) (3.5, 5.5) (3.586, 5.929)

    };

    \node[circle, fill=red, draw=none, line width=.2pt, inner sep=.8pt, label={above:\color{red}\footnotesize $\alpha$}] at (axis cs:1, 0.5) {};
    \node[circle, fill=orange, draw=orange, line width=.2pt, inner sep=.8pt, label={below:\color{orange}\footnotesize $\alpha^+$}] at (axis cs:1, -0.5) {};
    \node[circle, fill=orange, draw=orange, line width=.2pt, inner sep=.8pt, label={above:\color{orange}\footnotesize $\alpha^l$}] at (axis cs:0.618, -0.309) {};
    \node[circle, fill=orange, draw=orange, line width=.2pt, inner sep=0.8pt, label={ right:\color{orange}\footnotesize $\alpha^r$}] at (axis cs:1.382, 0.455) {};

\end{axis}
\end{tikzpicture}
\caption{The Dr\'ezet--Le Potier curve $\DLP$ on the plane $\mathbb{U}$
}
\label{fig:s,q}
\end{figure}
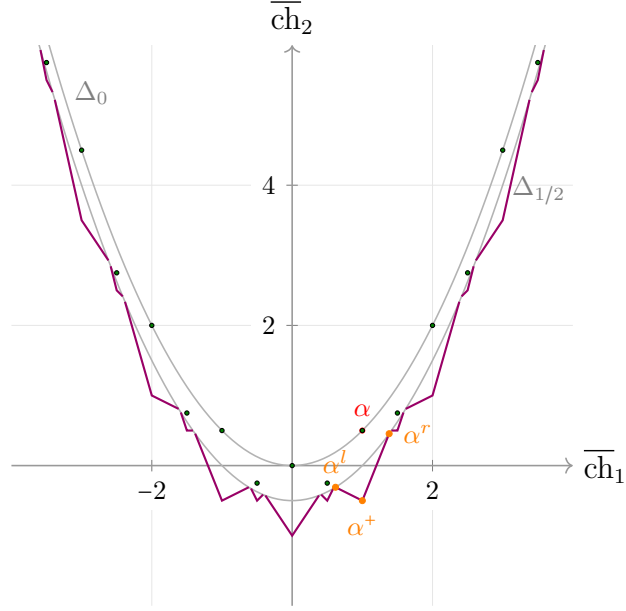

For $e \in K(\bP^2)$ satisfying $\ch_0(e)=0$, there exists a coherent sheaf with class $e$ if and only if
\begin{enumerate}
    \item $\ch_1(e) > 0$; or
    \item $\ch_1(e)=0$ and $\ch_2(e)>0$.
\end{enumerate}
For classes in Case (1), we set $\chbar_1(e)=\infty$ and view these classes as lying on a vertical line at infinity on the right side of $\bU$, with vertical coordinate given by $\ch_2(e)/\ch_1(e)$. In Case (1), the moduli space $M(e)$ is also irreducible of dimension $1-\chi(e,e)$ \cite[Thm. 1.1]{LeP93b}.

From the definition of $\DLP$, we obtain a result that is useful for ensuring there exist stable sheaves $G$ arising as cokernels in \eqref{eq:ext-small-strata}.

\begin{lem}\label{lem:stable-classes} Suppose $e$ is on or below $\DLP$ and assume $\alpha$ is an exceptional class
such that $r(\alpha) \le r(e)$, $\mu(\alpha) < \mu(e)$, and the slope of $L_{\alpha e}$ is $\le \mu(e)-3/2$. Then $e-\alpha$ is on or below $\DLP$. 
\end{lem}

\begin{proof}

For every exceptional class $\beta$, the line $L_{\beta^l \beta^+}$ can be described as $\chi(\beta,-)=0$ and has slope $\mu(\beta)-3/2$. 
Assume for contradiction that $e-\alpha$ is above $\DLP$. Then there is an exceptional class $\beta$ such that $e-\alpha$ lies above $l_{\beta^l \beta^+} \cup l_{\beta^+ \beta^r}$, while $e$ and $\alpha$ are left of $\beta$ and lie on or below $L_{\beta^l \beta^+}$, as shown in Figure~\ref{fig:stable-classes}.
But then the slope of $L_{\alpha e}$ is greater than the slope of $L_{\beta^l \beta^+}$, which contradicts the assumption that the slope of the former is $\le \mu(e)-3/2$.

\begin{figure}[h]
\centering
\begin{tikzpicture}
\pgfmathsetmacro{\mB}{1}
\pgfmathsetmacro{\sA}{-0.5}
\pgfmathsetmacro{\qA}{-0.25}
\pgfmathsetmacro{\sE}{0.7}
\pgfmathsetmacro{\qE}{-0.35}
\pgfmathsetmacro{\rA}{2}\pgfmathsetmacro{\rE}{10}
\pgfmathsetmacro{\qB}{\mB*\mB/2}
\pgfmathsetmacro{\qBp}{\qB-1}
\pgfmathsetmacro{\tt}{(3-sqrt(5))/2}
\pgfmathsetmacro{\sBl}{\mB-\tt}  \pgfmathsetmacro{\qBl}{\sBl*\sBl/2-0.5}
\pgfmathsetmacro{\sBr}{\mB+\tt}  \pgfmathsetmacro{\qBr}{\sBr*\sBr/2-0.5}
\pgfmathsetmacro{\sG}{(\rE*\sE-\rA*\sA)/(\rE-\rA)}
\pgfmathsetmacro{\qG}{(\rE*\qE-\rA*\qA)/(\rE-\rA)}
\pgfmathsetmacro{\ms}{(\qE-\qA)/(\sE-\sA)}
\begin{axis}[
    xmin=-1.6, xmax=1.9,
    ymin=-0.7, ymax=0.95,
    axis lines=middle,
    xtick=\empty, ytick=\empty,
    axis line style={draw=gray!80, line width=0.6pt, ->},
    xlabel={$\chbar_1$}, ylabel={$\chbar_2$},
    width=12cm, height=6cm,
    clip=false,
    every axis x label/.style={at={(current axis.right of origin)}, anchor=west},
    every axis y label/.style={at={(current axis.above origin)}, anchor=south}
]
    \addplot[domain=-1.414:1.414, samples=100, gray!70, line width=0.6pt] {x^2/2}
        node[pos=0.98, anchor=south east, gray, font=\footnotesize] {$\Delta_0$};
    \addplot[domain=-1.6:1.732, samples=100, gray!70, line width=0.6pt] {x^2/2-0.5}
        node[pos=0.02, anchor=north west, gray, font=\footnotesize] {$\Delta_{1/2}$};

    \addplot[domain=-1.6:1.9, samples=2, black, dashed, line width=0.7pt]
        {\qE+\ms*(x-\sE)};

    \draw[black, line width=1.1pt] (axis cs:\sBl,\qBl) -- (axis cs:\mB,\qBp);
    \draw[black, line width=1.1pt] (axis cs:\mB,\qBp) -- (axis cs:\sBr,\qBr);

    \node[circle, fill=black, inner sep=1.2pt, label={above:$\beta$}]
        at (axis cs:\mB,\qB) {};
    \node[circle, fill=black, inner sep=1.2pt, label={right:$\beta^r$}]
        at (axis cs:\sBr,\qBr) {};
    \node[circle, fill=black, inner sep=1.2pt, label={below right:$\beta^+$}]
        at (axis cs:\mB,\qBp) {};
    \node[circle, fill=black, inner sep=1.2pt, label={above left:$\beta^l$}]
        at (axis cs:\sBl,\qBl) {};
    \node[circle, fill=black, inner sep=1.2pt, label={above:$\alpha$}]
        at (axis cs:\sA,\qA) {};
    \node[circle, fill=black, inner sep=1.2pt, label={[label distance=-3pt]below left:$e$}]
        at (axis cs:\sE,\qE) {};
    \node[circle, fill=black, inner sep=1.2pt, label={[label distance=-3pt]above left:$g$}]
        at (axis cs:\sG,\qG) {};

    \node[anchor=north, font=\footnotesize]
        at (axis cs:{(\sBl+\mB)/2},{(\qBl+\qBp)/2 - 0.04}) {$l_{\beta^l \beta^+}$};
    \node[anchor=west, font=\footnotesize]
        at (axis cs:{\mB+0.22},{(\qBp+\qBr)/2}) {$l_{\beta^+ \beta^r}$};
    \node[anchor=south east, font=\footnotesize]
        at (axis cs:-1.15,{\qE+\ms*(-1.15-\sE)+0.03}) {$L_{\alpha e}$};
\end{axis}
\end{tikzpicture}
\caption{
}
\label{fig:stable-classes}
\end{figure}

\end{proof}

We extend the notion of $\oo(k)$-Segre strata (see Definition \ref{defn:F-Segre-stratum}) to moduli of sheaves of small ranks with the following definition.

\begin{defn}\label{defn:O(k)-segre-stratum-small-rks}
When $r(e)=0$ or $1$, the $\oo(k)$-\emph{Segre stratum} in $M(e)$ is the set of S-equivalence classes of semistable sheaves $E$ such that $\gr^{\mathrm{JH}}(E)$ admits a nonzero map from $\oo(k)$ but no nonzero map from $\oo(k+1)$. \end{defn}

Note that the stratum just defined is locally closed. Moreover, for $E$ in the $\oo(k)$-Segre stratum in the case $r(e)=1$, the nonzero map $\oo(k) \to E$ must be an inclusion, so we still call $
\oo(k)$ a \emph{Segre subsheaf} of $E$ and consider the \emph{Segre invariant} of $E$ to be $c_1(e)-k$.

\subsection{General rank-1 Segre invariant}

In this subsection, we assume $e \in K(\bP^2)$ has positive rank and that $M(e)$ is of positive dimension. We use a theorem of G\"{o}ttsche and Hirschowitz to identify the general Segre invariant. We then prove Theorem \ref{thm:O(k)-Segre-strata} for the general Segre stratum and explain some implications for general stable sheaves.

For any integer $k$, note that
\begin{equation}\label{eq:chi}
    \chi(e(-k))
    = \chi(e)+\frac{r(e)(k^2-3k)}{2}-c_1(e) k.
\end{equation}
As the right side of \eqref{eq:chi} makes sense for any $k \in \bR$, we sometimes apply the formula for real-valued $k$.
The solutions of the quadratic equation $\chi(e(-k))=0$ are $k_{\pm}=\mu(e)+\frac{3}{2} \pm \sqrt{1/4+2\Delta(e)}$.
Let
\begin{align}\label{eq:ke}
    k_e = \lceil k_- \rceil - 1,
\end{align}
namely $k_e$ is the largest integer less than $k_-$.

\begin{lem}\label{lem:k_e}
The integer $k_e$ defined by \eqref{eq:ke} is the maximal integer $k < \mu(e)$ satisfying $\chi(e(-k)) > 0$. Equivalently, $k_e$ is the unique integer $k < \mu(e)$ such that $\chi(e(-k))>0$ and $\chi(e(-k-1)) \le 0$.
\end{lem}

\begin{proof} First, we prove $k_e<\mu(e)$. Every $E \in M(e)$
satisfies $\ext^2(\oo(k_e),E)=\hom(E,\oo(k_e-3))=0$ by semistability.
Since $\chi(\oo(k_e),E)=\chi(e(-k_e))>0$, we must have $\hom(\oo(k_e),E)>0$, which implies $k_e < \mu(E)$ by semistability.
Then the first statement about $k_e$ in the lemma follows from $\mu(e) < k_+$ and the second statement follows from $k_e+1 < k_+$.
\end{proof}

\begin{lem}\label{lem:gen-segre-p2}
    Let $e$ be a class such that $r(e)\geq 1$ and $M(e)$ is of positive dimension. Then, the general Segre invariant of rank $1$ is $c_1(e)-r(e)k_e$ and corresponds to the subsheaf $\oo(k_e)$.
\end{lem}

\begin{proof}
     Suppose $E\in M(e)$ is a general sheaf and $F\subset E$ is a rank-$1$ Segre subsheaf. If $r(e)=1$, then $F$ is locally free by definition. If $r(e) \ge 2$, then since $E$ is general, it is locally free and so is $F$. Let $k$ be the integer such that $F\cong \oo(k)$.
     Then,
     the main theorem of \cite{GotHir98} implies that $\chi(E(-k))>0$.
    On the other hand, $E(-k_e)$ is also general and $\chi(E(-k_e))>0$ implies that $h^0(E(-k_e))=\chi(E(-k_e))>0$, according to the same theorem \cite{GotHir98}.
     Therefore, $k=k_e$, $\oo(k_e)$ is a Segre subsheaf of $E$, and the Segre invariant is $c_1(e)-r(e)k_e$. 
\end{proof}

Note that for general $E \in M(e)$, we have $\hom(\oo(k_e+m),E)=0$ for all $m>0$ by Lemma \ref{lem:gen-segre-p2}, and $\ext^2(\oo(k_e+m),E)=\hom(E,\oo(k_e+m-3))=0$ for $m \le 3$ by stability. Thus
\begin{equation}\label{eq:chi-non-pos}
    \chi([\oo(k_e+m)],e) \le 0 \qquad \text{for $0<m \le 3$}.
\end{equation}

A coherent sheaf $E$ over $\bP^2$ is called \emph{prioritary} if $\Ext^2(E,E(-1))=0$ \cite{hirschowitz1993fibres}. 
\begin{lem}
\label{lem:prior-ext}
Suppose $G$ is a coherent sheaf and $E$ is obtained as an extension \eqref{eq:ext-small-strata}. 
Then
\begin{enumerate}[(a)]
    \item If $G$ is prioritary and $\hom(G,\oo(k-2))=0$, then $\hom(G,E(-2))=0$.
    \item If also $\hom(\oo(k+2),G)=0$, then $E$ is prioritary.
\end{enumerate}
\end{lem}
We omit the proof, since it's similar to the corresponding argument in the proof of Proposition \ref{prop:bn}.

We now complete Step 1 of the proof of Theorem \ref{thm:O(k)-Segre-strata} for the case $k=k_e$. As irreducibility and nestedness follow immediately in this case, this completes the proof of the theorem for $k=k_e$.

\begin{proof}[Proof of Theorem \ref{thm:O(k)-Segre-strata}, the case $k=k_e$]
Let $g = e - [\oo(k_e)]$. In the case $r(e)=1$, the result follows from the fact that a general $E \in M(e)$ is a twist of an ideal sheaf of general points, hence the cokernel $G$ of a general nonzero map $\oo(k_e) \to E$ is a general line bundle supported on a smooth curve, which is stable.

Now assume $r(e)\geq 2$. Note that $M(e)$ of positive dimension implies $\chi(e,[\oo(k_e)]) \le 0$ by \eqref{eq:chi-non-pos} and Serre duality.
Then $\chi(g,[\oo(k_e)])=\chi(e,[\oo(k_e)])-1 <0$.
Thus, a general $G \in M(g)$ is torsion-free and satisfies $\ext^1(G,\oo(k_e))=-\chi(G,\oo(k_e))$ by \cite{GotHir98}. Thus the $\Ext^1(G,\oo(k_e))$-family over an open subset of $M(g)$ is a vector bundle.
On the other hand, notice that $\chi(\oo(k_e),G)=\chi(E(-k_e))-1 \ge 0$ and thus $\Ext^1(\oo(k_e),G)=0$ for general $G$. According to Lemma \ref{lem:complete}, the family of general extensions \eqref{eq:ext-small-strata} provides a complete family of sheaves $E$. 
By stability, $\Hom(G,\oo(k_e-2))=0$. By \eqref{eq:chi-non-pos}, $\chi([\oo(k_e+2)],g)=\chi([\oo(k_e+2)],e) \leq 0$.
Thus, for general $G$ of class $g$, $\Hom(\oo(k_e+2),G)=0$.
By Lemma \ref{lem:prior-ext}, $E$ is prioritary.
Therefore, a general extension \eqref{eq:ext-small-strata} provides a stable $E$.  
\end{proof}

Let $E\in M(e)$ be general, assuming $r(e)\ge 1$. We now show that $E$ can be expressed as an extension of a general ideal sheaf by a direct sum of line bundles. In the case $r(e)=1$, $E$ is already an ideal sheaf. If $r(e) \ge 2$, by Theorem \ref{thm:O(k)-Segre-strata}, $\oo(k_e)\subset E$ is a Segre subsheaf and the quotient $G_1:=E/\oo(k_e)$ is general in $M(g_1)$, where $g_1=e-[\oo(k_e)]$. 
Moreover, we have the exact sequence
\begin{align*}
    0\to \Hom(\oo(k_{g_1}), \oo(k_e)) \to \Hom(\oo(k_{g_1}),E)\to \Hom(\oo(k_{g_1}), G_1) \to 0. 
\end{align*}
Therefore, an inclusion $\oo(k_{g_1})\hookrightarrow G_1$ can be lifted to an inclusion $\oo(k_{g_1})\hookrightarrow E$ that does not factor through $\oo(k_e)$. Thus, we have a subsheaf $\oo(k_{g_1}) \oplus \oo(k_e) \subset E$. 
Similarly, $\oo(k_{g_1})$ is a Segre subsheaf of $G_1$ and the quotient $G_2:=G_1/\oo(k_{g_1})$ is general in $M(g_2)$, where $g_2=g_1-[\oo(k_{g_1})]$.
Continuing in this way, we can obtain a sequence of integers
\begin{equation}\label{eq:int-seq}
    k_e\geq k_{g_1}\geq \dots \geq k_{g_{r(e)-2}},
\end{equation}
depending only on $e$, and a subsheaf
\begin{equation}\label{eq:int-seq-sub}
\oo(k_{g_{r(e)-2}}) \oplus \cdots \oplus \oo(k_{g_1}) \oplus \oo(k_e) \subset E
\end{equation}
such that the quotient is rank 1 and torsion-free, hence is a twist of an ideal sheaf $I_Z$ for $Z$ a general 0-dimensional subscheme. 

In fact, we see in the proof of the next result that the first $\chi(e(-k_e))$ of the integers in \eqref{eq:int-seq} are equal to $k_e$ and all remaining integers are equal to $k_e-1$. The case $r(e)=1$ of the result, describing resolutions of general ideal sheaves, was known to Gaeta \cite{Gae51} (see also \cite{Hui16-interpolation}).

\begin{cor}\label{cor:gen-res} A general $E \in M(e)$ admits a resolution
\[
    0 \to \oo(k_e-2)^{a_2} \to \oo(k_e-1)^{a_1} \oplus \oo(k_e)^{a_0} \to E \to 0
\]
or
\[
    0 \to \oo(k_e-2)^{a_2} \oplus \oo(k_e-1)^{-a_1} \to \oo(k_e)^{a_0} \to E \to 0,
\]
depending on whether $a_1 \ge 0$ or $a_1 \le 0$. Here $a_0=\chi(e(-k_e))$, $a_1=\chi(e(-k_e+1))-3a_0$, and $a_2=a_0+a_1-r(e)$.
\end{cor}

\begin{proof} Write $\ch(e)=(r,c,d)$. We may assume after twisting that $k_e=0$. Then $c>0$,
\[
    a_0 = \chi(e)=r+\tfrac{3}{2}c+d>0, \qquad \text{and} \qquad \chi(e(-1))=\tfrac{1}{2}c+d \le 0.
\]
Note that $\chi(e(1))=3r+\tfrac{5}{2}c+d$, so $a_1=-2c-2d$.

\emph{Case 1:} Suppose $a_0 \ge r$. Then the integers in \eqref{eq:int-seq} are all 0, so $E$ admits a short exact sequence $0 \to \oo^{r-1} \to E \to I_Z(c) \to 0$ in which $Z$ is general. Then $\chi(I_Z(c))>0$, so the Gaeta resolution for $I_Z(c)$ involves the line bundles $\oo(-2),\oo(-1),\oo$. There are two cases for the Gaeta resolution, depending on whether $a_1 \ge 0$ or $a_1 \le 0$, and combining this resolution with the inclusion $\oo^{r-1} \hookrightarrow E$ gives the result.

\emph{Case 2:} Suppose $a_0 < r$. Note that $a_0+a_1=r-\tfrac{1}{2}c-d \ge r$, so the first $a_0$ integers in the sequence \eqref{eq:int-seq} are 0 and the remaining integers are all $-1$. Then $E$ admits a short exact sequence $0 \to \oo(-1)^{r-1-a_0} \oplus \oo^{a_0} \to E \to I_Z(c') \to 0$, where $I_Z(c')$ has vanishing cohomology and thus the length of $Z$ equals $\chi(\oo(c'))$. It follows that the Gaeta resolution for $I_Z(c')$ involves only the line bundles $\oo(-2)$ and $\oo(-1)$, giving the result.
\end{proof}

Existence of these resolutions in certain cases was proven in \cite{CosHuiWoo17} using the generalized Beilinson spectral sequence  (e.g. see \cite[Prop. 5.10]{BerGolJoh23} for a precise statement in the case $a_1 \ge 0$). Corollary \ref{cor:gen-res} expands on this range of cases.

The subsheaf \eqref{eq:int-seq-sub} yields an upper bound for the general Segre invariant of any rank $r'$.

\begin{cor}\label{cor:gen-Segre-bound}
The general Segre invariant of rank $r^\prime$ is not greater than $r^\prime c_1(e)-r(e)r'k_e$ if $r' \le \chi(e(-k_e))$ or $r' c_1(e)-r(e)(r'(k_e-1)+\chi(e(-k_e)))$ if $r' > \chi(e(-k_e))$.
\end{cor}

\subsection{Bounds on Euler characteristics}

This subsection introduces key calculations related to the expected dimensions of Segre strata. Roughly speaking, the aim is to prove the following statement. Let $e$ be a class such that $r(e) \ge 0$ and $M(e)$ is positive-dimensional, let $k$ be an integer satisfying $k_e < k < \mu(e)$, and let $f$ be a class of rank $r'>0$ in the region bounded below by $L_{\oo(k)e}$, above by $\Delta_0$, and on the right by the vertical line through $e$.
Then
\begin{equation}\label{eq:euler-comparison}
    \chi(f,e-f) \le \chi(e(-k))-1,
\end{equation}
with equality if and only if $f=[\oo(k)]$. This statement can be interpreted as saying that the expected dimension of the locus in $M(e)$ of sheaves admitting $\oo(k)$ as a subsheaf is strictly larger than that of the locus of sheaves admitting a subsheaf of any other class $f$ in the region. 
As this statement is not true in full generality, we state three precise results, based on whether $r'<r(e)$, $r'>r(e)$, or $r'=r(e)$.

First, consider the case $r'<r(e)$, which implies $r(e) \ge 2$.
Consider the following region in the plane. Let $f_0 = e-(r(e)-r')(1,\mu(e),\mu(e)^2/2)$, which is the point directly below $e$ such that $e-f_0$ lies on $\Delta_0$. Then let $R_{r',k}(e)$ denote the closed region bounded above by $\Delta_0$, below by $L_{\oo(k)f_0}$, and on the right by the vertical line through $e$. See Figure \ref{fig:region-R}.

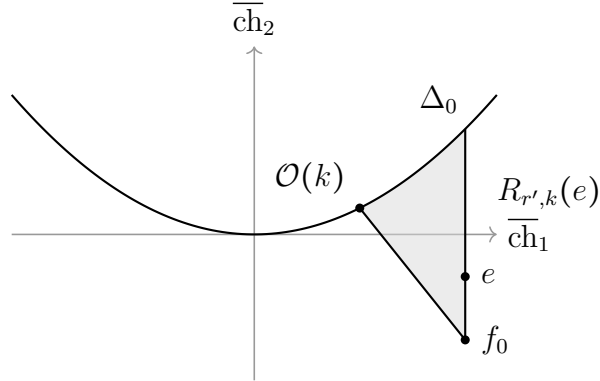
\begin{figure}[h]
\centering
\begin{tikzpicture}
\pgfmathsetmacro{\sK}{0.5}
\pgfmathsetmacro{\sE}{1}
\pgfmathsetmacro{\qF}{-0.5}
\pgfmathsetmacro{\qE}{-0.2}
\pgfmathsetmacro{\qK}{\sK*\sK/2}
\pgfmathsetmacro{\slope}{(\qF-\qK)/(\sE-\sK)}
\begin{axis}[
    xmin=-1.15, xmax=1.15,
    ymin=-0.6, ymax=0.8,
    axis lines=middle,
    axis equal,
    xtick=\empty, ytick=\empty,
    axis line style={draw=gray!80, line width=0.6pt, ->},
    xlabel={$\chbar_1$},
    ylabel={$\chbar_2$},
    width=8cm, height=6cm,
    clip=false,
    every axis x label/.style={at={(current axis.right of origin)}, anchor=west},
    every axis y label/.style={at={(current axis.above origin)}, anchor=south}
]
    \addplot[draw=none, name path=upper, domain=\sK:\sE, samples=60] {x^2/2};
    \addplot[draw=none, name path=lower, domain=\sK:\sE, samples=2]
        {\qK + \slope*(x-\sK)};
    \addplot[gray!30, opacity=0.5] fill between[of=upper and lower];

    \addplot[domain=-1.15:1.15, samples=100, black, line width=0.8pt] {x^2/2}
        node[pos=0.93, anchor=south east] {$\Delta_0$};

    \draw[black, line width=0.8pt] (axis cs:\sK,\qK) -- (axis cs:\sE,\qF);
    \draw[black, line width=0.8pt] (axis cs:\sE,{\sE*\sE/2}) -- (axis cs:\sE,\qF);

    \node[circle, fill=black, inner sep=1.2pt, label={above left:$\oo(k)$}] at (axis cs:\sK,\qK) {};
    \node[circle, fill=black, inner sep=1.2pt, label={right:$e$}]           at (axis cs:\sE,\qE) {};
    \node[circle, fill=black, inner sep=1.2pt, label={right:$f_0$}]         at (axis cs:\sE,\qF) {};

    \node[anchor=west] at (axis cs:{\sE+0.1},{\sE*\sE/2 - 0.3}) {$R_{r',k}(e)$};
\end{axis}
\end{tikzpicture}
\caption{The region $R_{r',k}(e)$ in the $(\chbar_1,\chbar_2)$-plane, bounded above by $\Delta_0$, below by $L_{\oo(k)f_0}$, and on the right by the vertical line through $e$.}
\label{fig:region-R}
\end{figure}

\begin{lem}\label{lem:chi-bound} Let $e$ be a class such that $r(e)\ge 2$ and $M(e)$ is positive-dimensional. Let $r'$ denote an integer satisfying $0 < r' < r(e)$. Then \eqref{eq:euler-comparison} holds for all classes $f$ in $R_{r',k}(e)$ satisfying $r(f)=r'$,
with equality if and only if $r'=1$ and $f=[\oo(k)]$.
\end{lem}

\begin{proof}
Writing $\ch(e)=(r,c,d)$, $\ch(f)=(r',c',d')$, and $g=e-f$, we have
\begin{align*}
    \chi(f,g) = r'(r-r') + \tfrac{3}{2}(r'c-rc') - c'(c-c') + r'd + (r-2r')d'.
\end{align*}
We prove the result by considering two cases: $r' \le r/2$ and $r/2 < r' < r$.

\medskip

\emph{Case 1:} Suppose $r' \le r/2$. Then $\chi(f,g)$ increases or is constant as $d'$ increases, hence an upper bound is obtained by taking $f$ on $\Delta_0$. Setting $f = r' [\oo(j)]$ for $j \ge k$, we have $\chi(f,g) = r' \chi(e(-j))-(r')^2$. As a function of $j$, this parabola has a minimum when $j=c/r+3/2$, and as we are considering only $j \le c/r$, $\chi(f,g)$ increases as $j$ decreases, hence $\chi(f,g)$ is bounded above by its value when $j=k$. Moreover, since both terms in the expression for $\chi(f,g)$ are non-positive, $\chi(f,g)$ increases as $r'$ decreases. Thus $\chi(f,g)$ is bounded above by its value when $f=[\oo(k)]$, giving the result.

\medskip

\emph{Case 2:} Suppose $r/2 < r' < r$. Then $\chi(f,g)$ increases as $d'$ decreases, so an upper bound for $\chi(f,g)$ is obtained by taking $f$ on $L_{\oo(k)f_0}$. We can parametrize the points on $l_{\oo(k)f_0}$
by $f(t) = (1-t)f_0 + tr'[\oo(k)]$, $t\in [0,1]$. Then $\chi(f(t),e-f(t))$ is quadratic as a function of $t$. Since the rank of $f(t)$ is independent of $t$, the $t^2$ term comes from $c_1(f(t))^2$ and has positive coefficient since $c/r \ne k$.
Hence, the maximum value of $\chi(f(t),e-f(t))$ must occur when $t=0$ or $1$. 
If $f = r' [\oo(k)]$, then, as discussed above, $\chi(f,g)$ is bounded by its value when $r'=1$, giving the desired bound.
If $f=f_0$, then a straightforward calculation
shows that $\chi(f_0,e-f_0)=(r-r')\chi(e(-c/r))-(r-r')^2$.
Since $\chi(e(-c/r))<\chi(e(-k))$, the result follows.
\end{proof}

Next, for the case $r'>r$,
consider the following region. Let $(a,a^2/2)$ denote the left intersection point of $L_{\oo(k)e}$ with $\Delta_0$, where $a = (2\ch_2(e)-c_1(e)k)/(c_1(e)-kr(e))$.
Let $f_{a}$ denote the class of rank $r'$ such that $e-f_{a}$ lies at this intersection point. Define $S_{r',k}(e)$ to be the closed region in the plane bounded below by $L_{\oo(k)e}$, above by $\Delta_0$, and on the right by the vertical line through $f_a$. See Figure \ref{fig:region-S}.

\begin{figure}[h]
\centering
\begin{tikzpicture}
\pgfmathsetmacro{\sK}{1.0}
\pgfmathsetmacro{\sA}{-2.5}
\pgfmathsetmacro{\sF}{1.75}
\pgfmathsetmacro{\sE}{2.4}
\pgfmathsetmacro{\qK}{\sK*\sK/2}
\pgfmathsetmacro{\qA}{\sA*\sA/2}
\pgfmathsetmacro{\slope}{(\sK+\sA)/2}
\pgfmathsetmacro{\xmaxx}{3.0}
\begin{axis}[
    xmin=-\xmaxx, xmax=\xmaxx,
    ymin=-1, ymax=4.2,
    axis lines=middle,
    xtick=\empty, ytick=\empty,
    axis line style={draw=gray!80, line width=0.6pt, ->},
    xlabel={$\chbar_1$},
    ylabel={$\chbar_2$},
    width=6cm, height=8cm,
    clip=false,
    every axis x label/.style={at={(current axis.right of origin)}, anchor=west},
    every axis y label/.style={at={(current axis.above origin)}, anchor=south}
]
    \addplot[draw=none, name path=upper, domain=\sK:\sF, samples=60] {x^2/2};
    \addplot[draw=none, name path=lower, domain=\sK:\sF, samples=2]
        {\qK + \slope*(x-\sK)};
    \addplot[gray!40, opacity=0.5] fill between[of=upper and lower];

    \addplot[domain=-2.75:2.75, samples=120, black, line width=0.8pt] {x^2/2}
        node[pos=0.95, anchor=south east] {$\Delta_0$};

    \addplot[domain=-2.9:2.9, samples=2, black, dashed, line width=0.7pt]
        {\qK + \slope*(x-\sK)}
        node[pos=0.3, anchor=south west] {$L_{\oo(k)e}$};

    \addplot[domain=\sK:\sF, samples=40, black, line width=1.1pt] {x^2/2};
    \draw[black, line width=1.1pt] (axis cs:\sK,\qK)
        -- (axis cs:\sF,{\qK + \slope*(\sF-\sK)});
    \draw[black, line width=1.1pt] (axis cs:\sF,{\sF*\sF/2})
        -- (axis cs:\sF,{\qK + \slope*(\sF-\sK)});

    \node[circle, fill=black, inner sep=1.2pt,
          label={left:$(a,a^2/2)$}]        at (axis cs:\sA,\qA) {};
    \node[circle, fill=black, inner sep=1.2pt,
          label={left:$\oo(k)$}]               at (axis cs:\sK,\qK) {};
    \node[circle, fill=black, inner sep=1.2pt,
          label={below left:$f_a$}]             at (axis cs:\sF,{\qK + \slope*(\sF-\sK)}) {};
    \node[circle, fill=black, inner sep=1.2pt,
          label={above right:$e$}]                   at (axis cs:\sE,{\qK + \slope*(\sE-\sK)}) {};

    \node[anchor=west] at (axis cs:{\sF+0.12},{(\sF*\sF/2 + \qK + \slope*(\sF-\sK))/2})
        {$S_{r',k}(e)$};
\end{axis}
\end{tikzpicture}
\caption{The region $S_{r',k}(e)$ in the $(\chbar_1,\chbar_2)$-plane, bounded below by $L_{\oo(k)e}$, above by $\Delta_0$, and on the right by the vertical line through $f_a$.}
\label{fig:region-S}
\end{figure}
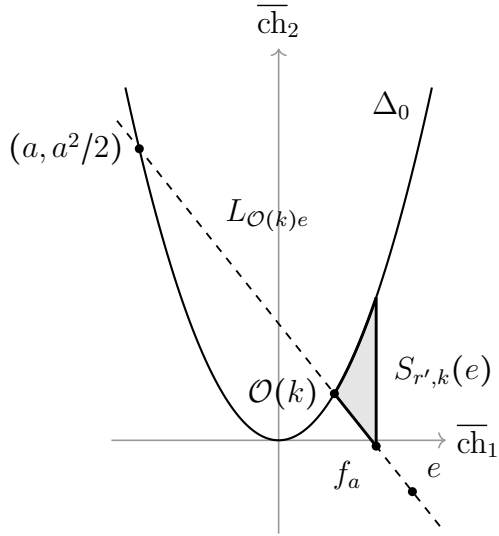

\begin{lem}\label{lem:chi-bound-large-rank} Let $e$ be a class such that $r(e) \ge 0$ and $M(e)$ is positive-dimensional. Let $k$ be an integer satisfying $k_e < k < \mu(e)$ and let $r'>r(e)$ be an integer. Then \eqref{eq:euler-comparison} holds for all classes $f$ in $S_{r',k}(e)$ satisfying $r(f)=r'$,
with equality if and only if $r(e)=0$, $r'=1$ and $f=f_a$.
\end{lem}

\begin{proof} Write $e=(r,c,d)$, $f=(r',c',d')$, and $g=e-f$.
Since $\chi(f,g)$ increases as $d'$ decreases, it suffices to find a bound for $f$ lying on $l_{\oo(k)f_a}$.
Similar to the argument in the proof of Case 2 of Lemma \ref{lem:chi-bound}, the maximum value of $\chi(f,g)$ is bounded by its value when $f$ is an endpoint of the segment, namely when $g$ is a multiple of
$(1,a,a^2/2)$ or when $f$ is a multiple of $[\oo(k)]$. In the latter case, we already observed in the proof of Lemma \ref{lem:chi-bound} that $\chi(f,g)$ is bounded above by $\chi(e(-k))-1$. For the former case,
we have $g = (r-r')(1,a,a^2/2)$ and
\[
    \chi(f,g)=\chi(e-g,g) = -(r'-r)^2 -(r'-r) \chi(e^*(a)).
\]
Note that $\chi(e^*(a)) = \chi(e(-a-3))$, so it suffices to prove
\begin{equation}\label{eq:suff-ineq}
    -\chi(e(-a-3)) \le \chi(e(-k)),
\end{equation}
as this ensures $\chi(e(-a-3))$ is non-negative and thus we may reduce to the case $r'=r+1$, giving the result.

In the case $r=0$, $\chi(e(-j)) = \frac{3}{2}c+d-cj$ and $a = 2d/c-k$, and a simple calculation shows $\chi(e(-a-3)) = -\chi(e(-k))$, giving (\ref{eq:suff-ineq}).
In the case $r>0$, as a function of $j$, $\chi(e(-j))$ is an upward parabola with roots $k_\pm=\mu+\frac{3}{2} \pm \sqrt{\frac{1}{4}+2\Delta}$, where we write $\Delta=\Delta(e)$ and $\mu=\mu(e)$. Moreover, $k$ satisfies $k_- \le k < \mu$. To deduce (\ref{eq:suff-ineq}), it suffices to prove
\[
    k_- - (a+3) \ge k-k_-,
\]
as the parabola is steeper to the left of $k_-$. Using
the identity $a = \mu-2\Delta/(\mu-k)$, this inequality is equivalent to
\[
    (\mu-k)^2 - 2 \sqrt{\tfrac{1}{4}+2\Delta}\,(\mu-k) + 2\Delta \ge 0.
\]
The left side is quadratic in $\mu-k$ and has roots $\sqrt{1/4+2\Delta} \pm 1/2 = \{\,\mu-k_-+1, \mu-k_-+2\,\}$. As $\mu-k \le \mu-k_-$, $\mu-k$ cannot lie between these roots, which concludes the proof.
\end{proof}

In the case $r'=r(e)$, \eqref{eq:euler-comparison} may fail for $f$ of rank $r'$ above $L_{\oo(k)e}$. However, the inequality can be restored by introducing the expected dimension of the $\oo(k)$-Segre stratum in $M(f)$ or $M(g)$ as a correction term. This is motivated by the observation that for a sheaf $E$ admitting a short exact sequence \eqref{eq:ext}, $\hom(\oo(k),E)>0$ implies $\hom(\oo(k),F)>0$ or $\hom(\oo(k),G)>0$.

\begin{lem}\label{lem:codimension-estimates-r=r'} Assume $\ch_0(e) \ge 1$ and $k_e < k < \mu(e)$. Suppose $f$ is a class of rank $r(e)$ that lies in the region bounded below by $L_{\oo(k)e}$, above by $\Delta_0$, and on the right by the vertical line through $e$ (not including the right boundary). Set $g=e-f$. Then
\begin{enumerate}[(a)]
\item $\chi(e(-k))> \chi(f(-k))+\chi(f,g)$; and
\item if $\ch_1(g)/2+\ch_2(g)/\ch_1(g) \ge k$, then $\chi(e(-k)) > \chi(g(-k))+\chi(f,g)$.
\end{enumerate}
\end{lem}

\begin{proof}
Write $e=(r,c,d)$ and $f=(r,c',d')$. For simplicity, we may assume after twisting that $k=0$.
Then $c>0$ and $L_{\oo e}$ is the line $q=(d/c)s$, so by the assumption that $f$ lies in the region, $0 < c' < c$ and
\begin{equation}\label{eq:d'-lower-bound}
    d' \geq (\tfrac{d}{c})c'.
\end{equation}
Moreover, since $k_e < 0$, we have $\chi(e) \le 0$, namely \begin{equation}\label{eq:d-upper-bound}
-\tfrac{d}{c}-\tfrac{3}{2} \geq \tfrac{r}{c} > 0.
\end{equation}
Note also that
\[
    \chi(f,g)=\tfrac{3}{2}r(c-c')-c'(c-c')-r(d'-d).
\]
Thus, for (a), (\ref{eq:d'-lower-bound}) implies that
\begin{align*}
    \chi(e)-\chi(f)-\chi(f,g)
    &= \tfrac{3}{2}(c-c')-(d'-d)-\tfrac{3}{2}r(c-c')+c'(c-c')+r(d'-d) \\
    &= -\tfrac{3}{2}(r-1)(c-c')+c'(c-c')+(r-1)(d'-d) \\
    &\geq -\tfrac{3}{2}(r-1)(c-c')+c'(c-c')-(r-1)\tfrac{d}{c}(c-c') \\
    &= (c-c')(c'+(r-1)(-\tfrac{d}{c}-\tfrac{3}{2})).
\end{align*}
As $c-c' > 0$ and $c' > 0$, the last expression is positive by (\ref{eq:d-upper-bound}), which proves (a).

For (b), combining the assumption $\ch_1(g)/2+\ch_2(g)/\ch_1(g) \ge 0$ with (\ref{eq:d'-lower-bound}) yields
\begin{equation}\label{eq:c-c'-lower-bound}
    \tfrac{c-c'}{2} \ge \tfrac{d'-d}{c-c'} \ge -\tfrac{d}{c}.
\end{equation}
Thus, using (\ref{eq:c-c'-lower-bound}) and (\ref{eq:d'-lower-bound}),
\begin{align*}
    \chi(e)-\chi(g)-\chi(f,g)
    &= r + \tfrac{3}{2}c'+d'-\tfrac{3}{2}r(c-c')+c'(c-c')+r(d'-d) \\
    &= r + d' + c'(c-c'+\tfrac{3}{2}) + r(c-c')(\tfrac{d'-d}{c-c'}-\tfrac{3}{2}) \\
    &\geq r + (\tfrac{d}{c})c' + c'(-2\tfrac{d}{c}+\tfrac{3}{2}) + r(c-c')(-\tfrac{d}{c}-\tfrac{3}{2}) \\
    &= r + c'(-\tfrac{d}{c}+\tfrac{3}{2})+r(c-c')(-\tfrac{d}{c}-\tfrac{3}{2}),
\end{align*}
which is positive since every term is positive by (\ref{eq:d-upper-bound}).
\end{proof}

\section{Rank-one subsheaves over $\bP^2$: Bridgeland stability conditions}\label{sec:bridgland}

In this section, we use wall-crossing for Bridgeland moduli spaces over $\bP^2$ as a tool to study $\oo(k)$-Segre strata when $k>k_e$. The general notion of stability conditions was introduced by Bridgeland in the seminal work \cite{Bri07}. Stability conditions have been successfully applied in the study of the birational geometry of moduli spaces, see for instance \cite{ABCH13,BayMac14a,BayMac14b,Toda14,NuerYos20}. Over $\bP^2$, works related to wall-crossing that are particularly relevant to our study are \cite{BerMarWan14,CosHuiWoo17,LiZhao19}. As we find the setup used by Li and Zhao \cite{LiZhao19} to be the most convenient, we base our notation and setup on their work.
All these works over $\mathbb{P}^2$ are built upon the foundational work by Dr\'ezet and Le Potier \cite{DreLeP85}.

\subsection{Bridgeland stability conditions}

Let $\Db(\bP^2)$ denote the bounded derived category of coherent sheaves on $\bP^2$. Given a real number $s$, the following full subcategories of $\Coh(\bP^2)$ form a torsion pair:
\begin{align*}
    \cT_{s} &=\left\{\, E\in \Coh(\bP^2)\,\middle\vert \,\mu(G)>s
    \mbox{ for every quotient } G \mbox{ of }E\, \right\}\mbox{ and }\\
    \cF_{s} &=\left\{\, E\in \Coh(\bP^2)\,\middle\vert \,\mu(F)\leq s
    \mbox{ for every subsheaf } F \mbox{ of }E\, \right\}.
\end{align*}
The heart of the corresponding tilted $t$-structure is the full subcategory of $\Db(\bP^2)$ defined by
\begin{equation*}
    \Coh^{\#s}=\left\{\,E^\bullet\in \Db(\bP^2) \,\middle\vert\,
    \begin{array}{l}
        \cH^{-1}(E^\bullet)\in \cF_s,   \\
        \cH^0(E^\bullet)\in \cT_s,\\
        \cH^i(E^\bullet)=0 \mbox{ if }i\not=0\mbox{ or }-1
    \end{array}
    \,\right\},
\end{equation*}
where $\cH^i$ denote the cohomology sheaves.
Note that if $E$ is a slope-stable coherent sheaf, then $E \in \Coh^{\#s}$ if $\ch_0(E)=0$ or if $\mu(E) > s$, while $E[1] \in \Coh^{\#s}$ if $\mu(E) \le s$.

Given $(s,q)\in \bR^2$,
we define a function on $\Db(\bP^2)$ by
\begin{align*}
    Z_{s,q}(E)=(-\ch_2(E)+q \cdot\ch_0(E))+i(\ch_1(E)-s\cdot \ch_0(E)).
\end{align*}
If $(s,q)$ is in the set
\[
    \mathrm{Geo}_{\mathrm{DLP}} = \{(s,q) \in \bU \mid \text{$(s,q)$ is above $\DLP$ and not on $l_{\alpha\alpha^+}$ for any exceptional $\alpha$}\},
\]
then
\[\sigma_{s,q}=(\Coh^{\# s}, Z_{s,q})\] is a \emph{geometric stability condition} on $\Db(\bP^2)$.
When referring to a stability condition $\sigma = \sigma_{s,q}$, we always assume $(s,q) \in \mathrm{Geo}_{\mathrm{DLP}}$. For each nonzero object $E$ in $\Coh^{\#s}$, the \emph{phase} of $E$ is a real number in the interval $(0,1]$ defined as
\[
    \phi_{s,q}(E)=\tfrac{1}{\pi} \mathrm{Arg}(Z_{s,q}(E)).
\]
We say $E$ is \emph{$\sigma_{s,q}$-semistable} if for every subobject $F \hookrightarrow E$ in $\Coh^{\#s}$, $\phi_{s,q}(F) \le \phi_{s,q}(E)$. If this inequality is strict for all subobjects, then $E$ is \emph{$\sigma_{s,q}$-stable}.

The following fact provides a geometric interpretation of phase comparisons.

\begin{lem}[\cite{LiZhao19} Lem. 1.19 \& 1.20]\label{lem:same-slope}
    Let $P=(s,q) \in \Geo_{\mathrm{DLP}}$ and $E$ and $F$ be nonzero objects in $\Coh^{\#s}$. Then $\phi_{s,q}(F) < \phi_{s,q}(E)$ if and only if the slope of $L_{PF}$ is less than the slope of $L_{PE}$. In particular, $\phi_{s,q}(F)=\phi_{s,q}(E)$ if and only if $P$, $F$, and $E$ are collinear.
\end{lem}
In view of the lemma and the definition of $\mathrm{Geo}_{\mathrm{DLP}}$, we consider points $(s,q)$ inducing stability conditions and the reduced Chern characters inside the same real plane $\bU$. We always view $\chbar_1$ or $s$ as giving the horizontal axis and $\chbar_2$ or $q$ as giving the vertical axis, and use terms such as ``above'' or ``to the right of'' accordingly.
We sometimes denote a point $(s,q)\in \Geo_{\mathrm{DLP}}$ as $\sigma_{s,q}$ or just $\sigma$. Recall also Convention \ref{conv:non-reduced}.
See Figure \ref{fig:s,q} for the $(s,q)$-plane $\bU$.

\begin{exmp}[\cite{LiZhao19} Corollary 1.22]\label{ex:exc-stable} Suppose $\alpha$ is an exceptional class and $E_\alpha$ is the corresponding exceptional bundle. Let $\sigma$ be a stability condition. If $\sigma$ is left of $\alpha$, then $E_\alpha$ is $\sigma$-stable. If $\sigma$ is above or right of $\alpha$, then $E_\alpha[1]$ is $\sigma$-stable. 
\end{exmp}

\begin{rem}\label{rem:bridgeland-lpcurve} Theorem \ref{thm:lpcurve} extends to the setting of Bridgeland stability. Namely, if $E$ is $\sigma$-semistable for some $\sigma$, then $E$ is proportional to an exceptional class or lies on or below $\DLP$ \cite[Corollary 1.33]{LiZhao19}.
\end{rem}

The following key fact ensures that stability of an object $E$ is unchanged as $\sigma$ varies along a line through $E$, as long as $\sigma$ does not cross a segment $l_{\alpha \alpha^+}$. We include the proof, as it is instructive.

\begin{lem}[\cite{LiZhao19} Corollary 1.24] \label{lem:no-change-along-wall} Let $\sigma=\sigma_{s,q}$ and $E \in \Coh^{\#s}$ be $\sigma$-semistable. Then for any $\sigma'=\sigma_{s',q'}$ on $L_{\sigma E}$ such that $l_{\sigma' \sigma}$ is contained in $\Geo_{\mathrm{DLP}}$, $E$ is $\sigma'$-semistable.
\end{lem}

\begin{proof} First, we prove $E$ is in $\Coh^{\#s'}$. Let $F$ denote a maximum slope subsheaf of $\cH^{-1}(E)$ and $G$ denote a minimum slope quotient of $\cH^0(E)$. Then $F$ and $G$ are slope-semistable sheaves, $\mu(F) \le s < \mu(G)$, $F[1]$ is a subobject of $E$, and $G$ is a quotient object of $E$. By $\sigma$-semistability of $E$, $\sigma$ lies on or below the lines $L_{FE}$ and $L_{GE}$.
As $\Geo_{\mathrm{DLP}}$ contains no points on or directly below $F$ and $G$ and $l_{\sigma'\sigma}$ is contained in $\Geo_{\mathrm{DLP}}$, we must have $\mu(F) < s' < \mu(G)$, hence $E \in \Coh^{\#s'}$. Now suppose for contradiction that $E$ is $\sigma'$-destabilized by a $\sigma'$-semistable subobject $E'$. By the first part of the proof, $E' \in \Coh^{\#s}$. But then $E'$ $\sigma$-destabilizes $E$, contradiction.
\end{proof}

Given a full strong exceptional sequence $\mathcal{E}=(E_1, E_2, E_3)$ of coherent sheaves in $\bP^2$, the full abelian subcategory $\mathcal{A}_{\mathcal{E}} = \langle E_1[2],E_2[1],E_3\rangle$ of $\Db(\bP^2)$ is equivalent to the category of finite-dimensional quiver representations of the Beilinson quiver for $\bP^2$. Set $\alpha_i = [E_i]$. The region $\mathrm{MZ}_{\mathcal{E}}$ bounded by the line segments $l_{\alpha_1 \alpha_1^+}$, $l_{\alpha_1^+ \alpha_2}$, $l_{\alpha_2 \alpha_3^+}$, $l_{\alpha_3^+ \alpha_3}$, and $l_{\alpha_3 \alpha_1}$ lies above $\DLP$.
Given $e \in K(\bP^2)$ with $\ch_0(e) \ge 0$ and a stability condition $\sigma$ to the left of $e$, the line $L_{e\sigma}$ passes through $\mathrm{MZ}_{\mathcal{E}}$ for some exceptional triple $\mathcal{E}$. Let $P = (s,q)$ be a point such that $l_{P\sigma}$ is contained in $\mathrm{Geo}_{\mathrm{DLP}}$.
There is a real number $0 < t < 1$ such that the full subcategory $\Coh_P^{(t,t+1]}$ generated by $\sigma_P$-stable objects with phases in $(t,t+1]$ is equal to $\mathcal{A}_{\mathcal{E}}$. Moreover, $\sigma_P$-(semi)stability is equivalent to King (semi)stability of quiver representations. Suppose $E$ is $\sigma_P$-semistable.
If $L_{Pe}$ is above $\alpha_3$, then $E$ is in $\mathcal{A}_{\mathcal{E}}$, while if $L_{Pe}$ is above $\alpha_1$, then $E[1]$ is in $\mathcal{A}_{\mathcal{E}}$ \cite[Lem. 2.1]{LiZhao19}.
Thus the \emph{Bridgeland moduli space}
\[
    M_{\sigma}(e)
\]
of $\sigma_P$-semistable objects in $\Coh^{\#s}$ with class $e$
can be constructed as a GIT quotient space of semistable quiver representations. We use $M^\stable_{\sigma}(e)$ to denote the locus of $\sigma_P$-stable objects, which is open in $M_{\sigma}(e)$.

The construction is similar for $\ch_0(e)<0$ and $\sigma$ to the right of $e$.
The natural isomorphism
\begin{equation}\label{eq:moduli-nat-iso}
    M_{\sigma_{s,q}}(e) \cong M_{\sigma_{-s,q}}(-\ch_0(e),\ch_1(e),-\ch_2(e)) 
\end{equation}
induced by $E \mapsto \mathcal{R}\lHom(E,\oo)[1] = E^{\vee}[1]$ allows this case to be reduced to the case $\ch_0(e) > 0$ and $s < \chbar_1(e)$ considered above, showing that
$M_{\sigma_{s,q}}(e)$ parametrizes shifted derived duals of objects in $M_{\sigma_{-s,q}}(-\ch_0(e),\ch_1(e),-\ch_2(e))$.

\begin{exmp}\label{exmp:dual-ideal-sheaf} When $e \in K(\bP^2)$ is on or below $\DLP$ and satisfies $\ch_0(e)=-1$, then $(-\ch_0(e),\allowbreak \ch_1(e),\allowbreak -\ch_2(e))$ is the Chern character of the twisted ideal sheaf $I_Z(\ch_1(e))$, where $Z$ is a zero-dimensional subscheme. Thus, for $\sigma_{s,q}$ to the right of $e$, $\sigma_{-s,q}$-stability of $I_Z(\ch_1(e))$ is equivalent to $\sigma_{s,q}$-stability of the shifted derived dual $I^{\vee}_Z(-\ch_1(e))[1]$, which is a two-term complex whose cohomology sheaves are $\cH^{-1} = \oo(-\ch_1(e))$ and $\cH^0 = \oo_Z$.
\end{exmp}

\subsection{Wall-crossing}

For the rest of this section, we assume $e \in K(\bP^2)$ is on or below $\DLP$ and consider stability conditions $\sigma$ above $\Delta_0$. The vertical line through $e$ divides the region above $\Delta_0$ into two subsets: if $\ch_0(e) \ge 0$, then we consider $\sigma$ left of the line; if $\ch_0(e) < 0$, we consider $\sigma$ right of the line.
In light of Lemma \ref{lem:no-change-along-wall}, we say that a non-vertical line $L$ through $e$ is a \emph{wall} for $e$ if $M_{\sigma}(e)$ changes as $\sigma$ crosses $L$. There are only finitely many walls for $e$, which together with the vertical line through $e$ divide the region above $\Delta_0$ into subsets called chambers \cite{Maciocia14}. Given two walls $L_1$ and $L_2$ for $e$, we say $L_1$ is \emph{above} $L_2$ (and $L_2$ is \emph{below} $L_1$) if the segment of $L_1$ with endpoints $L_1 \cap \Delta_0$ lies above the segment of $L_2$ with endpoints $L_2 \cap \Delta_0$.

We now describe how the moduli space changes at a wall. Suppose $\sigma$ lies on a wall
for $e$ (except the last wall, as discussed below) and let $\sigma_+$ and $\sigma_-$ denote stability conditions in the chambers above and below that wall. There are natural birational morphisms
\[\xymatrix{
     M_{\sigma_-}(e) \ar[dr]_{p_{\sigma_-}} && M_{\sigma_+}(e) \ar[dl]^{p_{\sigma_+}} \\
     & M_{\sigma}(e)
}\]
Note that the exceptional loci of $p_{\sigma_+}$ and $p_{\sigma_-}$ have the same image in $M_{\sigma}(e)$. The proof of this fact is similar to that of \cite[(3.3) Thm.]{Tha96}, via reducing to the local case where the group action is by a torus and then using \cite[(1.4) Prop.]{Tha96}. For variations of GIT, see also \cite{DolHu98}.
In the case $\ch_0(e) \ge 0$, for $\sigma=\sigma_{s,q}$ with $q$ sufficiently large, $M_{\sigma}(e)=M(e)$.
For $\sigma$ in each chamber, as described above, $M_{\sigma}(e)$ is birational to $M(e)$.
In cases when $M(e)$ is smooth, varying $\sigma$ induces the minimal model program of $M(e)$ \cite[Thm. 2.24]{LiZhao19}.

The \emph{last wall} for $e$ is the wall $L_e^{\mathrm{last}}=L_{\sigma e}$ such that $M_{\sigma_-}(e)$ is empty, which lies below all other walls for $e$. When $L_{\sigma e}$ is the last wall, $p_{\sigma_+} \colon M_{\sigma_+}(e) \to M_{\sigma}(e)$ is a fibration and may not be birational, and every object in $M_{\sigma}(e)$ is strictly $\sigma$-semistable.
For $\sigma$ above the last wall, $M^\stable_{\sigma}(e)$ is non-empty and dense in $M_{\sigma}(e)$.

The last wall can be computed as follows. There is a unique line $\tilde{L}_e$ through $e$ such that the two points $\tilde{L}_e \cap \Delta_{1/2}$ have $\chbar_1$-distance (the difference of $\chbar_1$-values) equal to 3.
There is a unique exceptional class $\gamma$
such that the right intersection point of $\tilde{L}_e$ and $\Delta_{1/2}$ lies between $\gamma^l$ and $\gamma^r$.
According to \cite[Thm. A \& Defn. 3.11]{LiZhao19}, the line $L_e^{\mathrm{last}}$ is then determined by three cases depending on whether $e$ is above, below, or on the line $L_{\gamma(-3)^+\gamma^+ }$ given by the equation $\chi(\gamma,-)=0$:
\begin{enumerate}[1)]
\item If $\chi(\gamma,e) > 0$, then $L_e^{\mathrm{last}}=L_{\gamma e}$;
    \item If $\chi(\gamma,e) < 0$, then $L_e^{\mathrm{last}}=L_{\gamma(-3) e}$;
    \item If $\chi(\gamma,e) = 0$, then $L_e^{\mathrm{last}}=L_{\gamma(-3)^+\gamma^+ }$.
\end{enumerate}
In all cases, the last wall is on or below $L_{\gamma e}$. See \cite[Fig. 9]{LiZhao19} for a picture of the different cases of the last wall.
Moreover, if $\alpha$ is any exceptional class above $L_e^{\mathrm{last}}$ and left of the vertical line through $e$, then $\chi(\alpha,e)\le 0$ \cite[Prop. 1.30]{LiZhao19}. 

The walls above the last wall are determined as follows. For any $\sigma$ above $L_e^{\mathrm{last}}$ and left of the vertical line through $e$, if $L_{\sigma e}$ is a wall for $e$, then there exists a class $f$ in the interior of $l_{\sigma e}$ such that $\ch_0(f) > 0$
and $f$ and $e-f$ are either proportional to exceptional classes or on or below $\DLP$ (see Lemma \ref{lem:destab-seq} below). This necessary condition for $L_{\sigma e}$ to be a wall can be strengthened to a sufficient condition by including the assumption that $\sigma$ is above the last walls of $f$ and $e-f$ \cite[Thm. 3.16 \& Lem. 3.18]{LiZhao19}.

\begin{rem}\label{rem:right-walls} When $\ch_0(e)<0$, the walls for $e$ contain stability conditions $\sigma$ that lie to the right of the vertical line through $e$, and we call the last wall $L_e^{\text{right-last}}$. By (\ref{eq:moduli-nat-iso}), these walls can be obtained from the walls for $(-\ch_0(e),\ch_1(e),-\ch_2(e))$ by the reflection $(s,q) \mapsto (-s,q)$ on $\bU$.
\end{rem}

The exceptional locus $Z_{\sigma_+}$ of $p_{\sigma_+}$ consists of S-equivalence classes of objects
that are strictly $\sigma$-semistable and $\sigma_-$-unstable. These objects can be described precisely as follows.

\begin{lem}\label{lem:destab-seq} Suppose $E \in M_{\sigma_+}(e)$ is $\sigma_-$-unstable. Then there is a class $f$ on the interior of $l_{\sigma e}$ and a non-split short exact sequence
\begin{equation}\label{eq:ses-destab}
    0 \to F \to E \to G \to 0
\end{equation}
in $\Coh^{\#\sigma}$ with $F \in M_{\sigma}^s(f)$ and $G \in M_\sigma(e-f)$. 
Moreover, $\ext^1(G,F) > 0$ and $\ext^i(G,F)=0$ for all $i \ne 1$.
\end{lem}

\begin{proof}
Since $E$ is $\sigma_+$-semistable and $\sigma_-$-unstable, there is a $\sigma$-stable subobject $F \to E$ with $\ch(F)=f$ on $L_{\sigma e}$. If $\ch_0(e) \ge 0$, then $\ch_0(f) > 0$
and $f$ is left of $e$, as otherwise $F$ would $\sigma_+$-destabilize $E$.
If $\ch_0(e) < 0$, then using (\ref{eq:moduli-nat-iso}), $\ch_0(f) < 0$ and $f$ is right of $e$. In each case, $f$ is on the interior of $l_{\sigma e}$.
The cokernel $G$ of $F \to E$ must also be $\sigma$-semistable, as otherwise $E$ would have a quotient object of smaller phase, which by continuity of phases contradicts the $\sigma_+$-semistability of $E$. Since $E$ is an extension of $\sigma$-semistable objects, it is strictly $\sigma$-semistable.
The vanishing of $\Ext^i(G,F)$ for $i<0$ or $i>2$ follows from the fact that near the wall, the $\sigma$-semistable objects $E$, $F$, and $G$ can be realized as representations of quivers. Namely, since $E$, $F$, and $G$ all lie on $L_{\sigma e}$, these three objects, or their shifts, are all contained in $\mathcal{A}_{\mathcal{E}}$. As $\mathcal{A}_{\mathcal{E}}$ is equivalent to a category of representations of the Beilinson quiver for $\bP^2$, all $\Ext^i$ between objects in the category vanish unless $0 \le i \le 2$.

If there is a nonzero map $G \to F$, then the image $I$ is both a subobject and a quotient-object of $E$, hence $[I]$ must lie on $L_{\sigma e}$. This contradicts the $\sigma_+$-semistability of $E$ unless $[I]$ is a multiple of $e$, but then the fact that $I$ is a subobject of $F$ contradicts the $\sigma$-stability of $F$.
Thus $\Hom(G,F)=0$.

The vanishing of $\Ext^2(G,F)$ follows from this vanishing for $\sigma$-stable objects \cite[Lem. 2.9]{LiZhao19} using the Jordan-H\"{o}lder filtration of $G$. 
Finally, $G$ cannot be a subobject of $E$ due to the $\sigma_+$-semistability of $E$, hence the extension is non-split and thus $\ext^1(G,F)>0$. \end{proof}

An analogous argument yields the following description of the exceptional locus $Z_{\sigma_-}$ of $p_{\sigma_-}$, which contains objects that are $\sigma_+$-unstable.

\begin{lem}\label{lem:new-stab-seq} Suppose $E \in M_{\sigma_-}(e)$ is $\sigma_+$-unstable. Then there is a class $f$ on the interior of $l_{\sigma e}$ and a non-split short exact sequence
\begin{equation}\label{eq:ses-newstable}
    0 \to G \to E \to F \to 0
\end{equation}
in $\Coh^{\#\sigma}$ with $F \in M_{\sigma}^s(f)$ and $G \in M_\sigma(e-f)$. 
Moreover, $\ext^1(F,G) > 0$ and all other $\Ext^i(F,G)$ vanish.
\end{lem}

For classes $f$ as in Lemma \ref{lem:destab-seq}, let
\[
    Z_{f,+} \subset M_{\sigma_+}(e)
\]
denote the locus of S-equivalence classes of objects $E$ admitting short exact sequences (\ref{eq:ses-destab}).
Then the exceptional locus of $p_{\sigma_+}$ is
\[
    Z_{\sigma_+} = \bigcup_{f} Z_{f,+},
\]
where the union is finite and over all classes $f$ as in Lemma \ref{lem:destab-seq}.
Similarly, the exceptional locus of $p_{\sigma_-}$, denoted $Z_{\sigma_-}$, is the union of loci $Z_{f,-} \subset M_{\sigma_-}(e)$
containing those objects admitting extensions (\ref{eq:ses-newstable}). 

In the following, when we describe a wall for $e$ as $L_{fe}$ using a specific class $f$, we mean not just that $f$ lies on the wall but also that $Z_{f,+}$ or $Z_{f,-}$ is non-empty, namely $f$ is as in Lemma \ref{lem:destab-seq} or Lemma \ref{lem:new-stab-seq}.

\begin{lem}\label{lem:loci-codim-estimates} Suppose $L_{fe}$ is a wall for $e$. If $Z_{f,+}$ is non-empty, then it is irreducible and its codimension in $M_{\sigma_+}(e)$ is $\ge -\chi(f,e-f)$. If $Z_{f,-}$ is non-empty, then it is irreducible and its codimension in $M_{\sigma_-}(e)$ is $\ge -\chi(e-f,f)$. Moreover, $-\chi(e-f,f) > -\chi(f,e-f)$.
\end{lem}

\begin{proof} Assume $g=e-f$ is not a multiple of an exceptional character. A straightforward calculation shows that
\begin{equation}\label{eq:dim-ext1-bundle}
    \chi(f,g) + \dim M(e) = (-\chi(g,f)-1) + \dim M(f) + \dim M(g).
\end{equation}
Assuming $Z_{f,+}$ is non-empty, by Lemma \ref{lem:destab-seq}, there is a non-empty open subset $U \subseteq M_{\sigma}^s(f) \times M_{\sigma}(g)$ parametrizing pairs $(F,G)$ such that $\ext^1(G,F) = -\chi(g,f)$.
Thus there is a projective bundle over $U$ with fibers $\bP\Ext^1(G,F)$. The non-empty open subset of this projective bundle for which the extensions are $\sigma_+$-semistable maps to $M_{\sigma_+}(e)$ with image $Z_{f,+}$. Thus $Z_{f,+}$ is irreducible, and as the right side of (\ref{eq:dim-ext1-bundle}) is the dimension of the projective bundle, $Z_{f,+}$ has codimension $\ge -\chi(f,g)$ in $M_{\sigma_+}(e)$. The argument for the irreducibility and codimension of $Z_{f,-}$ is similar. The last statement follows from
\[
\chi(g,f)=\chi(f,g(-3))=\chi(f,g)-3(\ch_0(f)\ch_1(e)-\ch_0(e)\ch_1(f)) < \chi(f,g).
\]
In the case when $g = n\alpha$ for $n$ a positive integer and $\alpha$ the Chern character of an exceptional bundle $E_\alpha$, the same argument works with $-\chi(g,f)-1$ in (\ref{eq:dim-ext1-bundle}) replaced by $n(-\chi(\alpha,f)-n)$, the dimension of the Grassmannian $\mathrm{Gr}(n,\Ext^1(E_{\alpha},F))$ parametrizing extensions of $E_{\alpha}^n$ by $F$ up to the $\GL_n$-action on $E_{\alpha}^n$.
\end{proof}

\begin{lem}\label{lem:exc-gives-wall} Assume $\ch_0(e) > 0$. Let $\alpha$ be an exceptional class above $L_e^{\mathrm{last}}$ and left of the vertical line through $e$, and suppose $\ch_0(\alpha) \le \ch_0(e)$. Then $L_{\alpha e}$ is a wall for $e$. Moreover, $Z_{\alpha,+}$ is non-empty and its codimension in $M_{\sigma_+}(e)$ is exactly $-\chi(\alpha,e-\alpha)$.
\end{lem}

\begin{proof} 
Since $\ch_0(\alpha) \le \ch_0(e)$ and $\chbar_1(\alpha)<\chbar_1(e)$, $\ch_0(e-\alpha) \ge 0$ and $e-\alpha$ lies to the right of $e$. Since $\alpha$ is above $L_e^{\mathrm{last}}$, it is also above $\tilde{L}_e$, which has slope $\mu(e)-\sqrt{5/4+2\Delta(e)}$.
Since $\Delta(e)>1/2$, this slope is $<\mu(e)-3/2$. Thus Lemma \ref{lem:stable-classes} applies and ensures that $e-\alpha$ lies on or below $\DLP$.

By \cite[Lem. 3.13]{LiZhao19},
$L_{\alpha e}$ is above the last wall for $e-\alpha$, so letting $\sigma$ denote a stability condition on $L_{\alpha e}$, $M_{\sigma}^s(e-\alpha)$ is non-empty. Moreover, $\chi(e-\alpha,\alpha)=\chi(e,\alpha)-\chi(\alpha,\alpha)<\chi(\alpha,e)-\chi(\alpha,\alpha) < 0$,
so for $G \in M_{\sigma}^s(e-\alpha)$, general extensions of $G$ by $E_\alpha$ are $\sigma_+$-stable \cite[Lem. 3.1]{LiZhao19}.
Thus $Z_{\alpha,+}$ is non-empty. Since $\chi(\alpha,e-\alpha)=\chi(\alpha,e)-\chi(\alpha,\alpha)<0$, general $G \in M_{\sigma}^s(e-\alpha)$ satisfy $\Hom(E_\alpha,G)=0$ \cite[Thm. 1.6]{CosHuiKop21},
so the map from the projective bundle over $M_{\sigma}^s(e-\alpha)$ with fibers $\Ext^1(G,E_\alpha)$ to $M_{\sigma}(e)$ is generically finite, giving the result about the codimension of $Z_{\alpha,+}$.
\end{proof}

Suppose $Z_{f,+}$ is non-empty. In order to compare exceptional loci from different walls, for any stability condition $\sigma'$ above the wall $L_{fe}=L_{\sigma e}$, we define a locus $Z_{f,\sigma'} \subset M_{\sigma'}(e)$ that loosely speaking is the intersection of $Z_{f,+}$ with $M_{\sigma'}(e)$. The rigorous definition is inductive.
Suppose $L_{\sigma_1 e}$ is the closest wall above $L_{\sigma e}$. Then $M_{{\sigma_1}_-}(e) = M_{\sigma_+}(e)$, so we let $Z_{f,{\sigma_1}_-} = Z_{f,+}$. Define
\begin{equation}\label{eq:induced-locus-wall}
    Z_{f,\sigma_1} = p_{{\sigma_1}_-}(Z_{f,{\sigma_1}_-} \setminus Z_{{\sigma_1}_-}) \subset M_{\sigma_1}(e),
\end{equation}
and let $Z_{f,{\sigma_1}_+} = p_{{\sigma_1}_+}^{-1}(Z_{f,\sigma_1})$. Now repeat this construction for the second wall above $L_{\sigma e}$, using $Z_{f,{\sigma_1}_+}$ in place of $Z_{f,+}$, and so on for each prior wall. Denote the locus in $M(e)$ obtained in this way by
\[
    Z_f \subset M(e),
\]
which is the locus of S-equivalence classes of semistable sheaves $E$ that admit a map $F \to E$ as in (\ref{eq:ses-destab})
and are not destabilized at any wall above $L_{fe}$.

Note that each of these loci $Z_{f,\sigma'}$ are locally closed, irreducible, and either empty or of the same dimension as $Z_{f,+}$.
From (\ref{eq:induced-locus-wall}), we see that $Z_{f,\sigma'}$ is empty if and only if there is a wall $\sigma_j$ between $\sigma$ and $\sigma'$ such that $Z_{f,{\sigma_j}_-}$ is contained in the exceptional locus $Z_{{\sigma_j}_-}$. Note also that each $Z_{f,\sigma'}$ is isomorphic to an open subscheme of $Z_{f,+}$ since the definition avoids the exceptional loci at the walls above $L_{fe}$.

We conclude the section with two results about walls for $e$ of small rank that will be useful in the next section.

\begin{lem}\label{lem:last-wall-bound-rk1} Suppose $e$ is a class on or below the Dr\'ezet--Le Potier curve $\DLP$ with $\ch_0(e)=1$. Then the $\chbar_1$-distance between the two points $L_e^{\mathrm{last}} \cap \Delta_0$ is $\le 3$.
\end{lem}

\begin{proof}
Let $\gamma$ be the corresponding exceptional to $e$. We consider the three cases for $L_e^{\mathrm{last}}$, which correspond to three cases for $\chi(\gamma,e)$.

First, suppose $\chi(\gamma,e) > 0$, so the last wall for $e$ is $L_{\gamma e}$. The inequality $\chi(\gamma,e)>0$ implies that $e$ lies above $L_{\gamma(-3)^+ \gamma^+}$. In fact, since $\chi(\gamma,e)$ is an integer and $\ch_0(e)=1$, $e$ cannot lie below the line $\chi(\gamma,-)=\ch_0(-)/\ch_0(\gamma)$, which coincides with $L_{\gamma(-3)\gamma}$ (so $e$ cannot lie in the shaded region in Figure \ref{fig:strip}).
\begin{figure}[h]
\centering
\begin{tikzpicture}
\pgfmathsetmacro{\sG}{1}
\pgfmathsetmacro{\dd}{1}
\pgfmathsetmacro{\xminn}{-2.6}
\pgfmathsetmacro{\xmaxx}{2.6}
\pgfmathsetmacro{\qG}{\sG*\sG/2}
\pgfmathsetmacro{\sGt}{\sG-3}
\pgfmathsetmacro{\qGt}{\sGt*\sGt/2}
\pgfmathsetmacro{\slope}{(2*\sG-3)/2}
\begin{axis}[
    xmin=\xminn, xmax=\xmaxx,
    ymin=-1.6, ymax=3.6,
    axis lines=middle,
    xtick=\empty, ytick=\empty,
    axis line style={draw=gray!80, line width=0.6pt, ->},
    xlabel={$\chbar_1$},
    ylabel={$\chbar_2$},
    width=11cm, height=7cm,
    clip=false,
    every axis x label/.style={at={(current axis.right of origin)}, anchor=west},
    every axis y label/.style={at={(current axis.above origin)}, anchor=south}
]
    \addplot[draw=none, name path=upperline, domain=\xminn:\xmaxx, samples=2]
        {\qG + \slope*(x-\sG)};
    \addplot[draw=none, name path=lowerline, domain=\xminn:\xmaxx, samples=2]
        {\qG - \dd + \slope*(x-\sG)};
    \addplot[gray!40, opacity=0.5] fill between[of=upperline and lowerline];

    \addplot[domain=-2.68:2.68, samples=120, black, line width=0.8pt,
             restrict y to domain=-1.6:3.6] {x^2/2}
        node[pos=0.95, anchor=south east] {$\Delta_0$};

    \addplot[domain=\xminn:\xmaxx, samples=2, black, dashed, line width=0.9pt]
        {\qG + \slope*(x-\sG)};
    \addplot[domain=\xminn:\xmaxx, samples=2, black, dashed, line width=0.9pt]
        {\qG - \dd + \slope*(x-\sG)};

    \node[circle, fill=black, inner sep=1.2pt,
          label={[label distance=-2pt]above right:$\gamma(-3)$}]
        at (axis cs:\sGt,\qGt) {};
    \node[circle, fill=black, inner sep=1.2pt,
          label={[label distance=-2pt]below:$\gamma(-3)^+$}]
        at (axis cs:\sGt,{\qGt-\dd}) {};
    \node[circle, fill=black, inner sep=1.2pt,
          label={[label distance=-2pt]above:$\gamma$}]
        at (axis cs:\sG,\qG) {};
    \node[circle, fill=black, inner sep=1.2pt,
          label={[label distance=-2pt]below:$\gamma^+$}]
        at (axis cs:\sG,{\qG-\dd}) {};

    \node[anchor=south west, font=\footnotesize]
        at (axis cs:1.7,{\qG + \slope*(1.7-\sG)})
        {$\chi(\gamma,-)=\ch_0(-)/\ch_0(\gamma)$};
    \node[anchor=north east, font=\footnotesize]
        at (axis cs:0.2,{\qG - \dd + \slope*(0.2-\sG) - 0.15})
        {$\chi(\gamma,-)=0$};
\end{axis}
\end{tikzpicture}
\caption{$e$ does not lie inside the shaded strip between
$L_{\gamma(-3)\gamma}$ and $L_{\gamma(-3)^+\gamma^+}$}
\label{fig:strip}
\end{figure}
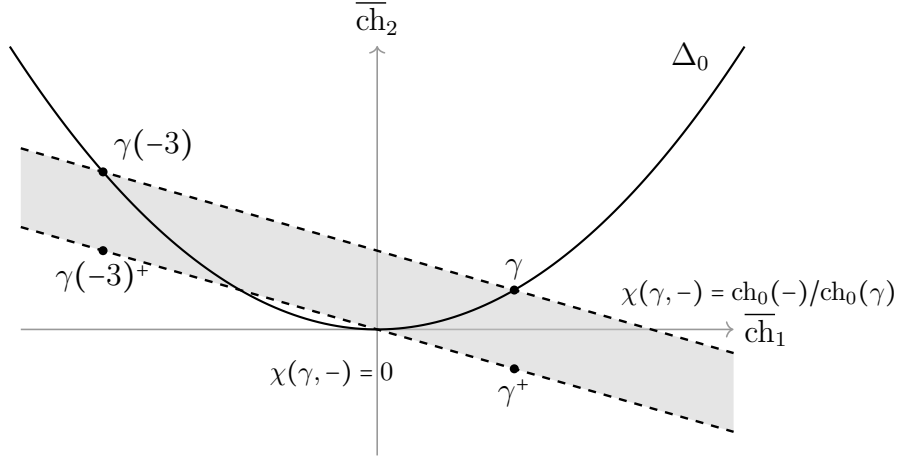
Note that
$L_{\gamma(-3)\gamma} \cap \Delta_{0}$ has $\chbar_1$-distance $\le 3$ since $\gamma$ is on or below $\Delta_0$. Since $e$ is on or above $L_{\gamma(-3)\gamma}$, $L_{\gamma e}$ has larger slope than $L_{\gamma(-3)\gamma}$,
hence $L_{\gamma e} \cap \Delta_0$ has $\chbar_1$-distance no larger than the $\chbar_1$-distance of $L_{\gamma(-3)\gamma} \cap \Delta_0$, giving the result.

Second, suppose $\chi(\gamma,e) < 0$, so that the last wall is $L_{\gamma(-3)e}$. Since $e$ is below the line $L_{\gamma(-3)\gamma}$, it follows that $L_{\gamma(-3)e} \cap \Delta_0$ has smaller $\chbar_1$-distance than $L_{\gamma(-3)\gamma} \cap \Delta_0$, giving the result.  

Finally, suppose $\chi(\gamma,e)=0$, so that the last wall is $L_{\gamma(-3)^+\gamma^+}$. This line is below $L_{\gamma(-3)\gamma}$, so $L_{\gamma(-3)^+\gamma^+} \cap \Delta_0$ has smaller $\chbar_1$-distance than $L_{\gamma(-3) \gamma} \cap \Delta_0$, giving the result.
\end{proof}

When $Z_{f,+}$ is non-empty and $\ch_0(f) > \ch_0(e)$, then $L_{fe}$ is called a \emph{higher-rank wall}. The fact that $e-f$ has negative rank and must lie on or below $\Delta_0$ imposes a useful bound on the higher-rank walls.
The lemma below provides the bound in the case when $\ch_0(e)=0$; note that in this case, every wall is a higher-rank wall.
This bound also appeared in the unpublished work \cite[Lem. 7.3]{Woo1305}.

\begin{lem}\label{lem:rk0-wall-bound} Let $e$ be a Chern character with $\ch_0(e)=0$ and $\ch_1(e)>0$. Let $L_{fe}$ be a wall for $e$. Then the $\chbar_1$-coordinate of the right intersection point of $L_{fe}$ with $\Delta_0$ is less than or equal to
\begin{equation}\label{eq:rk0-bound-walls}
    \frac{\ch_1(e)}{2} + \frac{\ch_2(e)}{\ch_1(e)}.
\end{equation}
\end{lem}

\begin{proof} Suppose $L_{fe}$ is a wall for $e$ and write its equation as $q-a^2/2 = (\ch_2(e)/\ch_1(e))(s-a)$, where $(a,a^2/2)$ is the right intersection point of the line with $\Delta_0$. The left intersection point $Q$ of $L_{fe}$ and $\Delta_0$ satisfies $\chbar_1(Q) = 2\ch_2(e)/\ch_1(e)-a$. Since $\ch_0(f)>0$, $\ch_0(e-f)<0$, and $f$ and $e-f$ must lie on or below $\Delta_0$, we have $\chbar_1(f) \ge a$ and $\chbar_1(e-f) \le \chbar_1(Q)$. Combining these two inequalities yields
\[
    a \le \frac{\ch_1(e)}{2\ch_0(f)} + \frac{\ch_2(e)}{\ch_1(e)}\leq \frac{\ch_1(e)}{2} + \frac{\ch_2(e)}{\ch_1(e)}.
\]
\end{proof}

\subsection{Smaller Segre strata on $\bP^2$}\label{sect:smaller-strata}

The main goal of this subsection is to complete Step 1 of the proof of Theorem \ref{thm:O(k)-Segre-strata}.
We prove that the $\oo(k)$-Segre stratum in $M(e)$ contains a unique component of maximum dimension whose general member is a general extension \eqref{eq:ext-small-strata}. The key idea is to consider Bridgeland wall crossings for $M(e)$ up to the wall $L_{\oo(k)e}$. The $\oo(k)$-Segre stratum is the union of its intersections with loci $Z_f$ for classes $f$ on these walls. We will show that $Z_{\oo(k)}$, which is contained in the $\oo(k)$-Segre stratum, provides the largest contribution to the stratum and that the general sheaf in $Z_{\oo(k)}$ arises as an extension (\ref{eq:ext-small-strata}). The argument relies on key results bounding the dimension of $Z_f$ (Lemmas \ref{lem:chi-bound}, \ref{lem:chi-bound-large-rank}, and \ref{lem:codimension-estimates-r=r'}).

We first execute this strategy to prove an analogous result for the rank-0 case, which will be used in the proof of the theorem.
In the case where $\ch_0(e)=0$, write $\ch(e)=(0,c,d)$, where $c>0$.
Then $\chi(e(-k))=\frac{3}{2}c+d-ck$, so $k_e = \lceil d/c + 1/2 \rceil$.

\begin{prop}\label{prop:O(k)-segre-strata-rk0} For the rank-0 case $\ch(e)=(0,c,d)$ with $c>0$, the $\oo(k)$-Segre stratum in $M(e)$ is non-empty if and only if $k_e \le k \le c/2+d/c$. If these conditions on $k$ hold, then the $\oo(k)$-Segre stratum has a unique component of largest dimension. The general sheaf in this component is stable and arises as a general extension
\[
    0 \to \oo_C(k) \to E \to \oo_Z \to 0,
\]
where $C$ is a smooth curve of degree $c$ and $Z$ is a general collection of $d+c^2/2-kc$ points on $C$.
If $k=k_e$, then this component is open and dense in $M(e)$, while otherwise it has codimension $1-\chi(e(-k))$ in $M(e)$.
\end{prop}

\begin{proof} The case $k=k_e$ is well-known, as the general sheaf in $M(e)$ is of the form $\oo_C(D)$, the pushforward of a general line bundle on a smooth curve. The condition $\chi(e(-k_e))>0$ ensures that $\oo_C(D)$ admits a nonzero map $\oo(k_e) \to \oo_C(D)$, which must factor through $\oo_C(k_e)$, giving a short exact sequence $0 \to \oo_C(k_e) \to \oo_C(D) \to \oo_Z \to 0$ as claimed. 

Walls for $M(e)$ are parallel lines with slope $\frac{d}{c}$ containing classes $f$ with $\ch_0(f)>0$ such that $f$ and $e-f$ are both on or below $\Delta_0$. By Lemma \ref{lem:rk0-wall-bound}, no wall can lie above the line $L_{(b_{\max},\frac{1}{2}b_{\max}^2) e}$, where $b_{\max} = c/2+d/c$. Thus, if $k>b_{\max}$ and $\sigma_{s,q}$ is a stability condition on $L_{\oo(k)e}$, then $M_{\sigma}(e) = M(e)$ and $L_{\oo(k)e}$ is not a wall for $e$. In this case, assume for contradiction that there exists $E \in M(e)$ admitting a nonzero map $\oo(k) \to E$.
As $E$ is semistable, it has no zero-dimensional torsion, hence the kernel of this map in $\Coh(\bP^2)$ must be an ideal sheaf $I_Z(k-a)$, where $Z$ is a zero-dimensional subscheme and $a>0$. Such sheaves are in $\mathcal{F}_s$ for $k-a \le s < k$. Thus for $\sigma$ on $L_{\oo(k)e}$ sufficiently close to $\oo(k)$, $\oo(k) \to E$ is injective in $\Coh^{\#s}$, contradicting the fact that $L_{\oo(k)e}$ is not a wall for $e$. Thus we have proved that the $\oo(k)$-Segre stratum is empty for all $k > b_{\max}$.

Now suppose $k$ satisfies $k_e < k \le c/2+d/c$. First, we show that $L_{\oo(k)e}$ is above $L_e^{\mathrm{last}}$. The assumption $k>k_e$ implies $\chi([\oo(k)],e) \le 0$, so $e$ lies on or below the line $L_{\oo(k-3)^+ \oo(k)^+}$. Thus $\gamma$, the corresponding exceptional class to $e$, is either $[\oo(k)]$ or else is left of $[\oo(k)]$. If $\gamma$ is left of $[\oo(k)]$, then clearly $L_{\oo(k)e}$ is above $L_e^{\mathrm{last}}$. Similarly, if $\gamma = [\oo(k)]$, then $\chi([\oo(k)],e) \le 0$ implies that $L_{\oo(k)e}$ is above $L_e^{\mathrm{last}}$.

Second, we argue that $M_{\sigma}(g)$ is non-empty, where $g=e-[\oo(k)]$ and $\sigma$ is a stability condition on $L_{\oo(k)e}$.
By Example \ref{exmp:dual-ideal-sheaf}, $g$ is the class of $I_Z^{\vee}(k-c)[1]$,
where $Z$ is a zero-dimensional subscheme of length $d+c^2/2-kc$.
The condition $k>k_e$ implies $\chi(e(-k)) \le 0$, hence the slope of $L_{\oo(k)e}$, which is $d/c$, satisfies $d/c \le k-\frac{3}{2}$. Since the line through $\oo(k)$ with slope $k-\frac{3}{2}$ is $L_{\oo(k-3)\oo(k)}$, the bound on $d/c$ implies that $L_{\oo(k)e} \cap \Delta_0$ has horizontal distance $\ge 3$. By Lemma \ref{lem:last-wall-bound-rk1} applied to $(-\ch_0(g),\ch_1(g),-\ch_2(g))$ and Remark \ref{rem:right-walls}, this implies that $L_{\oo(k)e}$ is on or above the right-last wall for $g$, giving the result.

Next, we show that $L_{\oo(k)e}$ is a wall for $e$ and that $Z_{\oo(k)}$ is non-empty by explicitly constructing stable extensions.
Let $E$ be given by a general extension
\[
    0 \to \oo(k) \to E \to I_Z^{\vee}(k-c)[1] \to 0,
\]
where $Z$ is general to ensure $I_Z^{\vee}(k-c)[1]$ is $\sigma$-semistable. Then $E$ is $\sigma$-semistable.
Taking cohomology sheaves yields a long exact sequence
\[
    0 \to \cH^{-1}(E) \to \oo(k-c) \to \oo(k) \to \cH^0(E) \to \oo_Z \to 0.
\]
Since the extension is general, the map $\oo(k-c) \to \oo(k)$ is general, hence it is injective and has cokernel $\oo_C(k)$, where $C$ is a smooth curve.
Thus $\cH^{-1}(E)=0$, namely $E = \cH^0(E)$, and $E$ is an extension $0 \to \oo_C(k) \to E \to \oo_Z \to 0$. As $E$ is $\sigma$-semistable, it cannot contain zero-dimensional torsion.
Thus $C$ contains $Z$, the extension is general, and $E$ is a stable sheaf.
By $\sigma$-semistability, it follows that $E$ is $\sigma'$-stable for all $\sigma'$ above $L_{\oo(k)e}$. In particular, $E$ is $\sigma_+$-stable, so $L_{\oo(k)e}$ is a wall for $e$ with $Z_{\oo(k),+}$ non-empty. Moreover, stability of such $E$ ensures that $Z_{\oo(k)}$ is non-empty.

We conclude the proof by using dimension estimates to show that $Z_{\oo(k)}$ yields the biggest component of the $\oo(k)$-Segre stratum. Similar to the proof of Lemma \ref{lem:exc-gives-wall}, the construction above shows that $Z_{\oo(k)}$ has codimension $1-\chi(e(-k))$ in $M(e)$. The $\oo(k)$-Segre stratum is the union of $Z_{\oo(k)}$ and intersections of the stratum with $Z_f$ for each wall $L_{fe}$ above $L_{\oo(k)e}$. 
By Lemma \ref{lem:chi-bound-large-rank},
\[
    1-\chi(e(-k)) < -\chi(f,e-f).
\]
The equality is strict since $f$ lies strictly above $L_{\oo(k)e}$.
By Lemma \ref{lem:loci-codim-estimates}, the expression on the right is a lower bound on the codimension of $Z_{f,+}$ in $M_{\sigma'_+}(e)$, where $\sigma'$ is a stability condition on $L_{fe}$. Thus
\[
    \dim Z_{f,+} < \dim Z_{\oo(k),+} = \dim Z_{\oo(k)}.
\]
It follows that the intersection of $Z_f$ with the $\oo(k)$-Segre stratum has dimension less than $\dim Z_{\oo(k)}$, completing the proof.
\end{proof}

We now complete Step 1 of the proof of Theorem \ref{thm:O(k)-Segre-strata}, establishing that for a class $e$ of positive rank, the $\oo(k)$-Segre stratum contains a unique component of maximum dimension whose general member is a general extension \eqref{eq:ext-small-strata}.
The argument is similar to the proof of the proposition, but additional cases arise.  The possibility of walls $L_{fe}$ for which $\ch_0(f) < \ch_0(e)$ requires Lemma \ref{lem:chi-bound}, while walls with $\ch_0(f)=\ch_0(e)$ requires more careful dimension estimates, for which we will use Lemma \ref{lem:codimension-estimates-r=r'}, Proposition \ref{prop:O(k)-segre-strata-rk0}, and induction on $\ch_1(e)$.

\begin{proof}[Proof of Theorem \ref{thm:O(k)-Segre-strata}, Step 1, case $k>k_e$]

The same argument as in the proof of Proposition \ref{prop:O(k)-segre-strata-rk0} shows that $L_{\oo(k)e}$ is above $L_e^{\mathrm{last}}$.
Let $\sigma$ denote a stability condition on $L_{\oo(k)e}$. By Lemma \ref{lem:exc-gives-wall} and its proof, $L_{\oo(k) e}$ is a wall for $M(e)$, $Z_{\oo(k),+}$ is non-empty, and the general object in $Z_{\oo(k),+}$ is a $\sigma_+$-stable sheaf given by an extension (\ref{eq:ext-small-strata}). It remains to show that (1) the general extension is stable, hence $Z_{\oo(k)}$ has the stated codimension, and (2) the intersection of the $\oo(k)$-Segre stratum with the locus $Z_f$ for each wall $L_{fe}$ above $L_{\oo(k)e}$ is of strictly smaller dimension.

After tensoring sheaves in $M(e)$ by $\oo(-k)$, we may assume $k=0$.
We prove the result by using induction on $\ch_1(e)$ to show that for every wall $L_{fe}$ above $L_{\oo e}$, the loci in $Z_{f,\pm}$ of objects $E$ such that $h^0(E)>0$ have dimension $< \dim Z_{\oo,+}$. This completes the proof as follows. First, it implies by (\ref{eq:induced-locus-wall}) that $Z_{\oo} \subset M(e)$ is non-empty, as letting $\sigma'$ denote a stability condition on $L_{fe}$, $Z_{\oo,\sigma'_-}$ is too large to be contained in $Z_{\sigma'_-}$. Thus general extensions (\ref{eq:ext-small-strata}) are semistable and the locus $Z_{\oo}$ in $M(e)$ has codimension $1-\chi(e)$. In fact, these general extensions are stable, as a destabilizing subsheaf $F \hookrightarrow E$ with $\chbar(F)=\chbar(E)$ would contradict $\sigma_+$-stability of $E$. Second, it ensures that the intersection of $Z_{f}$ with the $\oo$-Segre stratum has dimension $< \dim Z_{\oo}$, hence the largest contribution to the $\oo$-Segre stratum comes from $Z_{\oo}$.

For the base case $\ch_1(e)=1$, $L_{\oo e}$ is the first wall of $M(e)$, so there is nothing to prove.
For $\ch_1(e)>1$, suppose $L_{fe}$ is a wall for $e$ above $L_{\oo e}$, which implies $\ch_1(f)>0$. Let $\sigma$ be a stability condition on $L_{\oo e}$ and $\sigma'$ be a stability condition on $L_{fe}$.
If $\ch_0(f) \ne \ch_0(e)$, then Lemmas \ref{lem:chi-bound} and \ref{lem:chi-bound-large-rank} imply that
\begin{equation}\label{eq:codim-comparison}
     1-\chi(e) < -\chi(f,g),
\end{equation}
where we write $g=e-f$. By Lemmas \ref{lem:loci-codim-estimates} and \ref{lem:exc-gives-wall}, $\dim Z_{f,+} < \dim Z_{\oo,+}$, which gives the result.
If instead $\ch_0(f) = \ch_0(e)$, we analyze the locus of objects $E$ admitting extensions \eqref{eq:ses-destab} and satisfying $h^0(E)>0$ by considering the following two cases.

\emph{Case 1:} $h^0(F)>0$. Lemma \ref{lem:codimension-estimates-r=r'} implies that
\begin{equation}\label{eq:codim-comparison-2}
    1-\chi(e) < (1-\chi(f)) - \chi(f,g).
\end{equation}
If $\chi(f) > 0$, then this inequality implies (\ref{eq:codim-comparison}), hence $\dim Z_{f,+}<\dim Z_{\oo,+}$. So assume $\chi(f) \le 0$. Then $f$ is not exceptional, $k_f < 0$, and $\ch_1(f)<\ch_1(e)$, so we can apply the inductive hypothesis to $f$.
By the inductive hypothesis, following reasoning as in the previous paragraph, the locus in $M_{\sigma'}(f)$ of objects $F$ satisfying $h^0(F)>0$ has codimension equal to $1-\chi(f)$. Thus, as in the proof of Lemma \ref{lem:loci-codim-estimates}, the codimension in $M_{\sigma_+'}(e)$ of the locus in $Z_{f,+}$ of objects $E$ admitting extensions (\ref{eq:ses-destab}) with $h^0(F)>0$ is bounded below by the right side of (\ref{eq:codim-comparison-2}), hence this locus has dimension $< \dim Z_{\oo,+}$.

\emph{Case 2:} $h^0(G)>0$.
We may assume $\ch_1(g)/2 +\ch_2(g)/\ch_1(g) \ge 0$, as otherwise, by Proposition \ref{prop:O(k)-segre-strata-rk0}, there are no $G$ in $M_{\sigma'}(g)=M(g)$ such that $h^0(G)>0$.
Then, by Lemma \ref{lem:codimension-estimates-r=r'},
\begin{equation}\label{eq:codim-comparison-3}
    1-\chi(e) < (1-\chi(g)) - \chi(f,g).
\end{equation}
If $\chi(g) > 0$, we again deduce \eqref{eq:codim-comparison}, so assume $\chi(g) \le 0$. As $k_{g} < 0$, Proposition \ref{prop:O(k)-segre-strata-rk0} and its proof show that the locus of objects $G$ such that $h^0(G)>0$ has codimension $1-\chi(g)$ in $M_{\sigma'}(g)$. As in the previous case, it follows that the locus in $Z_{f,+}$ of objects $E$ such that $h^0(G)>0$ has dimension $<\dim Z_{\oo,+}$.

Finally, the argument that the locus in $Z_{f,-}$ of objects $E$ such that $h^0(E)>0$ has dimension $< \dim Z_{\oo,+}$ is similar, using $-\chi(g,f) > -\chi(f,g)$ (Lemma \ref{lem:loci-codim-estimates}) and analogous cases as above. This completes the proof of Step 1 of the theorem.
\end{proof}

\section{Proof and consequences of main result}\label{sec:conseq}

In this section, we finish the proof of Theorem \ref{thm:O(k)-Segre-strata}
by completing Steps 2 and 3. We then discuss several corollaries of the theorem.

\subsection{$I_Z(k)$-Segre strata}

This subsection is devoted to studying the Segre strata where the Segre subsheaves are ideal sheaves of points. The main result of the subsection completes Step 2 of the proof of Theorem \ref{thm:O(k)-Segre-strata}.

The following technical lemma will be used to relate these Segre strata to the $\oo(k)$-Segre strata.

\begin{lem}\label{lem:rk1-sub-tf}
Suppose $E_0$ is locally free with a Segre subsheaf $\oo(k) \subset E_0$. Set $d = \hom(\oo(k),E_0)$. 
Let $E$ be the kernel of an elementary modification $\alpha \colon E_0 \twoheadrightarrow \oo_W$. 
Let $w$ denote the length of $W$.
Then:
\begin{enumerate}[(a)]
\item If $w \le d-1$, then $\oo(k)$ is a Segre subsheaf of $E$.
\item If $w > d-1$, then each Segre subsheaf of $E$ is $I_{Z_m}(k) \subset E$ for some $Z_m \subset W$ with $m \le w+1-d$.
\item If $w > d-1$ and $W$ and $\alpha$ are general,
then $I_{Z_{w+1-d}}(k) \subset E$ is a Segre subsheaf for each $Z_{w+1-d} \subset W$.
\end{enumerate}
\end{lem}

\begin{proof} 
We consider (a). By applying the functor $\Hom(-, E_0)$ to the short exact sequence $0\to I_{W^\prime}(j)\to \oo(j)\to \oo_{W^\prime}\to 0$, we see that any subsheaf $I_{W'}(j) \hookrightarrow E$ induces an inclusion $\oo(j) \hookrightarrow E_0$, so $j \le k$. 
Moreover, if $w \le d-1$, then $\Hom(\oo(k), E)\not=0$ and $\oo(k)$ is the Segre subsheaf of $E$.

We next consider (b) and (c).
Each nonzero map $\sigma \colon \oo(k) \to E_0$ induces a commutative diagram whose rows are short exact sequences
\begin{equation}\label{eq:double-dual-diagram}\xymatrix{
    0 \ar[r] & I_{W'}(k) \ar[r] \ar[d] & \oo(k) \ar[r] \ar[d]^{\sigma} & \oo_{W'} \ar[r] \ar@{^(->}[d] & 0 \\
    0 \ar[r] & E \ar[r] & E_0 \ar[r]^{\alpha} & \oo_W \ar[r] & 0
}\end{equation}
in which $\oo_{W'} \subseteq \oo_W$ is the image of $\alpha \circ \sigma$.
Conversely, given any injective map $I_{W'}(k) \to E$, the composition $I_{W'}(k) \to E \to E_0$ factors through some map $\oo(k) \to E_0$. Assuming the cokernel of $I_{W'}(k) \to E$ is torsion-free, the induced map $\oo_{W'} \to \oo_W$ is injective, yielding a diagram (\ref{eq:double-dual-diagram}) as above.

Let $\rho$ denote the rank of the map
\[
    \Hom(\oo(k),E_0) \to \Hom(\oo(k),\oo_W)
\]
induced by $E_0 \twoheadrightarrow \oo_W$. If $\rho<d$, then $\Hom(\oo(k),E) \ne 0$, hence $E$ has Segre subsheaf $\oo(k)$. If $\rho = d$, then there must be a $\sigma$ such that $\alpha \circ \sigma$ vanishes at $d-1$ points of $W$, hence $m \le w+1-d$, where $m$ denotes the length of $W^\prime$. We have proven (b).

Assume that $w>d-1$ and $W$ and $\alpha$ are general. Then for every $Z_d \subseteq W$, the map
\[
    \Hom(\oo(k),E_0) \to \Hom(\oo(k),\oo_{Z_{d}})
\]
induced by $E_0 \twoheadrightarrow \oo_{Z_{d}}$
is an isomorphism. It follows that $\alpha \circ \sigma$ vanishes at no more than $d-1$ points of $W$, hence the image $\oo_{W'}$ of $\alpha \circ \sigma$ is supported at $m \ge w+1-d$ points. Thus, there is no map $I_Z(k) \to E$ with $|Z| < w+1-d$. Given any $Z_{d-1} \subseteq W$, there is a unique $\sigma$ (up to scaling) vanishing at $Z_{d-1}$. The corresponding diagram (\ref{eq:double-dual-diagram}) yields a Segre subsheaf $I_{W \setminus Z_{d-1}}(k) \to E$. We have proven (c).
\end{proof}

Slope-stability behaves well under elementary modifications and implies stability. The following lemma provides control over the locus of strictly slope-semistable sheaves.

\begin{lem}\label{lem:strictly-ss} Every component of the locus in $M(e)$ of strictly slope-semistable sheaves has codimension $> 1-\chi(e(-k))$ for all integers $k$ such that $k_e <k<\mu(e)$.
\end{lem}

\begin{proof}
If $E$ is semistable and strictly slope-semistable, then it is destabilized by $F \hookrightarrow E$ where $[F]$ lies on or directly below $[E]$ and $0<r(f)<r(e)$. Let $G$ denote the cokernel. Since $E$ is semistable, $F$ and $G$ must be slope-semistable. By \eqref{eq:dim-ext1-bundle},
since the data of the extension is lost under S-equivalence, the locus in $M(e)$ of sheaves whose slope-Jordan-H\"{o}lder filtration has two factors of classes $f$ and $g$ has codimension at least
\[
    -\chi(f,g)-\chi(g,f)-1.
\]
In cases where $F$ or $G$ are strictly slope-semistable, the slope-Jordan-H\"{o}lder filtration of $E$ has additional factors and a similar argument shows that the locus of such $E$ has even larger codimension. Since $g$ must be on or below $\Delta_0$, the result follows from Lemma \ref{lem:chi-bound}.

\end{proof}

\begin{prop}\label{prop:rk-1-subsh-p2}
    Let $e$ be a class of rank $\ge 2$. Let $k \in \bZ$ satisfy $k_e < k < \mu(e)$, and let $m\in \bZ_{> 0}$.
    The $I_{Z_m}(k)$-Segre stratum in $M(e)$ is non-empty if and only if $\chi (e(-k))+m\leq 1$.
    In this case, the stratum contains a component whose general member is a slope-stable
sheaf $E$ admitting a general extension
    \begin{equation}
    \label{eq:extension-ideal-subsheaf}
    0 \to I_{Z_m}(k) \to E \to G \to 0, 
\end{equation} 
where  $Z_m\in X^{[m]}$ is general and $G$ is a general stable sheaf.
This component has codimension $1-\chi(e(-k))+(r-2)m$ in $M(e)$. Any other component of the $I_{Z_m}(k)$-Segre stratum has dimension at most that of this component or consists entirely of strictly slope-semistable sheaves. 
\end{prop}

\begin{proof}
Suppose $E$ is semistable and $I_{Z_m}(k) \subset E$ is a Segre subsheaf. Let $E_0=E^{**}$. By the diagram (\ref{eq:double-dual-diagram}), there is a short exact sequence $0 \to E \to E_0 \to \oo_W \to 0$, where $E_0$ is slope-semistable with Segre subsheaf $\oo(k) \subset E_0$ such that $Z_m \subseteq W$.
Lemma \ref{lem:rk1-sub-tf} implies that $m \le |W|+1-d$, where $d = \hom(\oo(k),E_0)$. Thus
\[
    m \le |W|+1-d \le |W|+1-\chi(E_0(-k)) = 1-\chi(E(-k)).
\]

Suppose that $\chi (e(-k))+m\leq 1$.
Let $e^\prime= e+(0,0,m)$.
Then, $\chi(e^\prime(-k))=\chi(e(-k))+m\leq 1$
and \[k_{e^\prime}\leq k<\mu(e^\prime)=\mu(e).\] 
According to Step 1 of the proof of Theorem \ref{thm:O(k)-Segre-strata},
a general extension $E^\prime$ of the form \eqref{eq:ext-small-strata} is stable and locally free and has a Segre subsheaf $\oo(k) \subset E^\prime$. By Lemma \ref{lem:strictly-ss}, $E'$ is slope-stable.
Let $E^\prime\twoheadrightarrow \oo_{Z_m}$ be a general elementary modification and $E$ be its kernel. 
Then, $E$ is slope-stable as well and fits in the following commutative diagram with exact rows and columns:
\begin{equation*}
\xymatrix{
    0 \ar[r] & I_{Z_m}(k) \ar[r] \ar@{^(->}[d]^{\iota} & E \ar[r] \ar@{^(->}[d] & G \ar[r] \ar@{=}[d] & 0 \\
    0 \ar[r] & \oo(k) \ar[r] \ar@{->>}[d] & E^\prime \ar[r] \ar@{->>}[d] & G \ar[r] & 0 \\
    & \oo_{Z_m} \ar@{=}[r] & \oo_{Z_m}.
}\end{equation*}
Since $\chi(\oo(k), G)=\chi(e(-k))+m-1\leq 0$ and $G$ is general, $H^0(G(-k))=0$ and $\hom(\oo(k),E^\prime)=1$.
By Lemma \ref{lem:rk1-sub-tf}, $I_{Z_m}(k)$ is a Segre subsheaf of $E$.
A dimension count of elementary modifications shows that the locus of such $E$ has codimension $1-\chi(e(-k))+(r-2)m$ in $M(e)$. We omit the details, as a similar calculation is explained below.

It remains to show that every component of the $I_{Z_m}(k)$-Segre stratum that contains slope-stable sheaves has dimension bounded by $\dim M(e) - 1 + \chi(e(-k))-(r-2)m$. Suppose $E$ is slope-stable and in the stratum. Then $E$ admits a short exact sequence $0 \to I_{Z_m}(k) \to E \to G \to 0$. The natural map from $I_{Z_m}(k)$ to the double dual $E_0$ of $E$ fits in the following commutative diagram with exact rows and columns:
\begin{equation}\label{eq:ext-diagram}\xymatrix{
    0 \ar[r] & I_{Z_m}(k) \ar[r] \ar@{^(->}[d]^{\iota} & E \ar[r] \ar@{^(->}[d] & G \ar[r] \ar@{^(->}[d] & 0 \\
    0 \ar[r] & \oo(k) \ar[r] \ar@{->>}[d] & E_0 \ar[r] \ar@{->>}[d] & G' \ar@{->>}[d] \ar[r] & 0 \\
    0 \ar[r] & \oo_{Z_m} \ar[r] & \oo_{W} \ar[r] & \oo_{W'} \ar[r] & 0,
}\end{equation}
where $W$ and $W'$ are zero-dimensional subschemes of $\bP^2$.
Note that $\ell=\length(W)$ satisfies $\ell \ge m$. The double dual $E_0$ is locally free, slope-stable, and has Segre subsheaf $\oo(k)$.
Let $e_0=[E_0]$. Over the intersection of the $\oo(k)$-Segre stratum in $M(e_0)$ with the locus of locally-free sheaves, consider the family $\Omega_{\ell}$ of elementary modifications $E_0 \twoheadrightarrow \oo_W$, where $W \in (\bP^2)^{[\ell]}$. There is a natural map $\Omega_{\ell} \to M(e)$ whose image contains $E$. The moduli of elementary modifications of $E_0$ has dimension $2\ell + (r-1)\ell$. Thus, by Step 1 of the proof of Theorem \ref{thm:O(k)-Segre-strata}, the dimension of any component of the $I_{Z_m}(k)$-Segre stratum is bounded by
\[
    \dim M(e_0)-1+\chi(e_0(-k))+(r+1)\ell.
\]
Since $\chi(e_0,e_0)=\chi(e,e)+2r\ell$ and $\chi(e_0(-k))=\chi(e(-k))+\ell$, this dimension bound is
\[
    \dim M(e) - 1 + \chi(e(-k))-(r-2)\ell.
\]
Since $\ell \ge m$ and the case $\ell=m$ is the locus of extensions \eqref{eq:extension-ideal-subsheaf} described in the proposition, this gives the result.
\end{proof}

\begin{cor}\label{cor:unique-segre}
    Let $E$ be as in Theorem \ref{thm:O(k)-Segre-strata} or Proposition \ref{prop:rk-1-subsh-p2}. Then the rank-1 Segre subsheaf of $E$ is unique except in the case where $k=k_e$
    and $\chi(e(-k_e))>1$.
\end{cor}
\begin{proof}
    In Theorem \ref{thm:O(k)-Segre-strata}, when $k=k_e$, 
  we have  $\hom(\oo(k_e),E)=\chi(E(-k_e))$ since $\hom(\oo(k_e),G)=\chi(\oo(k_e),G)=\chi(\oo(k_e),E)-1\geq 0$. So, $E$ has a unique Segre subsheaf $\oo(k_e)\subset E$ if and only if $\chi(e(-k_e))=1$. 
On the other hand, when $k>k_e$, $\hom(\oo(k),E)=1$ because $\chi(\oo(k),G)=\chi(\oo(k),E)-1< 0$ and thus $\hom(\oo(k),G)=0$. 
In particular, when $k>k_e$, $E$ has a unique rank-one Segre subsheaf $\oo(k)\subset E$.

Next, consider the case of Proposition \ref{prop:rk-1-subsh-p2}. 
When $r=2$, let $G\cong I_{Z_n}(c_1-k)$ for some $Z_n\in X^{[n]}$. 
    The condition $\chi(e(-k)) +m\le 1$ is equivalent to $n\geq \chi(\oo(c_1-2k))$. A general $Z_n$ does not lie on a curve of degree $c_1-2k$ and $\hom(\oo(k),I_{Z_n}(c_1-k))=0$.
     Notice that $\chi(G(-k))=\chi(E^{**}(-k))-1=\chi(e(-k))+m-1\leq 0$. 
     When $r>2$ and $G$ is general, $h^0(G(-k))=0$. 
    On the other hand, $\Hom(I_{Z_m}(k),G)\cong\Hom(\oo(k),G)$.
    Therefore, $\hom(I_{Z_m}(k),E)=1$ and $I_{Z_m}(k)\subset E$ is the only Segre subsheaf of rank 1.
\end{proof}

\subsection{Irreducibility and nestedness of $\oo(k)$-Segre strata}

In this subsection, we complete Step 3 of the proof of Theorem \ref{thm:O(k)-Segre-strata}, which finishes the proof. We have shown that the $\oo(k)$-Segre strata have a unique component of maximum dimension whose general element is a general extension. Similarly, the $I_{Z_m}(k)$-Segre strata have a component whose general element is a general extension, and all other components are smaller or contained in the locus of S-equivalence classes of strictly slope-semistable sheaves. We now show that the $\oo(k)$-Segre strata are irreducible and nested.

\begin{proof}[Proof of Theorem \ref{thm:O(k)-Segre-strata}, Step 3] The case $k=k_e$ is clear from Step 1 of the proof of Theorem \ref{thm:O(k)-Segre-strata}, so we assume $k>k_e$.
Let $\Phi_k$ denote the locus in $M(e)$ of equivalence classes of sheaves $E$ such that $\gr^{\mathrm{JH}}(E)$ admits a nonzero map from $\oo(k)$. 
Let $Q^{\text{ss}} = \Quot^{\rm ss}(\oo(-m)^N,e)$ denote the Quot scheme such that 
$M(e)=Q^{\rm ss}/\mkern-3mu/\GL(N,\mathbb{C})$, as in the proof of Lemma \ref{lem:segre-strata-loc-closed}. Taking $S=Q^{\text{ss}}$ in Proposition \ref{prop:degeneracy-locus-codim}, every component of the degeneracy locus $D_{j_0}(\Phi)$ of quotients $\oo(-m)^N \twoheadrightarrow E$ such that $E$ admits a nonzero map from $\oo(k)$ has codimension $\le 1-\chi(e(-k))$ in $Q^{\rm ss}$. By a similar argument as in the proof of Lemma \ref{lem:strictly-ss}, it follows that every component of $D_{j_0}(\Phi)$ contains a stable sheaf. Thus, as $\Phi_k$ is the image of $D_{j_0}(\Phi)$ under the modular map $Q^{\rm ss} \to M(e)$, we see that every component of $\Phi_k$ has codimension $\le 1-\chi(e(-k))$ and contains a stable sheaf.

Since $r(e) \ge 1$, every nonzero map $\oo(k) \to E$ is injective, so
$\Phi_k$ is the disjoint union of all $\oo(j)$-Segre strata (for $k \le j < \mu(e)$) and the intersection of $\Phi_k$ with all $I_{Z_m}(j)$-Segre strata (for $k<j<\mu(e)$, $0<m \le 1-\chi(e(-j))$).
By Step 1 of the proof of Theorem \ref{thm:O(k)-Segre-strata} and Proposition \ref{prop:rk-1-subsh-p2},
the only possible component in $\Phi_k$ of codimension $\le 1-\chi(e(-k))$ is the one containing general extensions \eqref{eq:ext-small-strata}. By the above, this must be the only component of $\Phi_k$, so $\Phi_k$ is irreducible and equals the closure of the maximum-dimension component of the $\oo(k)$-Segre stratum. Thus the $\oo(k)$-Segre stratum is irreducible and its closure contains the $\oo(j)$-Segre stratum for all $k \le j <\mu(e)$.
\end{proof}

\subsection{Consequences of Theorem \ref{thm:O(k)-Segre-strata}}

We discuss several interesting results that follow from Theorem \ref{thm:O(k)-Segre-strata}, including statements about both general and non-general stable sheaves. First, we deduce the main corollary stating that general extensions of a stable sheaf by a rank-one sheaf are stable.

\begin{proof}[Proof of Corollary \ref{cor:lange-p2}]
In each case, it suffices to prove that there is a stable extension \eqref{eq:ext}. For (a), let $e=g+[\oo(k)]$. Then $\mu(e)>k$ and $\chi(e(-k-1))=\chi(g(-k-1)) \le 0$, so $k\ge k_e$.
Since $M(e)$ is of positive dimension, Theorem \ref{thm:O(k)-Segre-strata} applies and gives the result.

For (b), let $e=g+[I_{Z_m}(k)]$. Then $\mu(e)>k$, $r(e) \ge 2$,
and $\chi(e(-k))+m=\chi(g(-k))+\chi(I_{Z_m})+m = \chi(g(-k))+1 \le 1$. The conditions $\chi(e(-k)) \le 0$ and $\mu(e)>k$ imply that $k>k_e$ and that $\Delta(e)>1$, hence $M(e)$ is of positive dimension.
Thus Proposition \ref{prop:rk-1-subsh-p2} applies and gives the result.
\end{proof}

\begin{rem}
    Theorem \ref{thm:O(k)-Segre-strata} and the degeneracy locus description of the $\oo(k)$-Segre stratum (Proposition \ref{prop:degeneracy-locus-codim}) can be used to calculate the class of the closure of the stratum using the Thom--Porteous formula. 
\end{rem}

The cases $k>k_e$ of Theorem \ref{thm:O(k)-Segre-strata} have interesting implications for the Brill--Noether strata. The $j$th \emph{Brill--Noether stratum} is the locus in $M^{\stable}(e)$ defined by
\begin{align*}
    B_j(e)=\{\, E\in M^\stable (e)\mid h^0(E)=j\,\}. 
\end{align*}

\begin{cor}\label{cor:BN}
    Let $X=\bP^2$ and $e$ be a class of rank $\geq 1$ such that $M(e)$ has positive dimension. If $0\leq k < \mu(e)$ and $\chi(e) < \chi(\oo(k))$, then
    $B_{\chi(\oo(k))}(e)$ is non-empty and contains a component of codimension at most $1-\chi(e(-k))$ in $M^{\stable}(e)$. In the case $k=0$, $B_1(e)$ is non-empty, irreducible, and of codimension $1-\chi(e)$ in $M^{\stable}(e)$.
\end{cor}

\begin{proof}

First, we show that $k > k_e$. If $\chi(e)\le 0$, then $k_e < 0$, so the claim is trivial. This includes the case $k=0$, as then $\chi(e)<1$. So we may assume $0<k<\mu(e)$ and $0 < \chi(e) < \chi(\oo(k))$.
It suffices to prove that $\chi(e(-k))\le 0$, namely
    $\chi(e)+r(e)(k^2-3k)/2-c_1(e)k \le 0$,
in the most difficult case $\chi(e)=\chi(\oo(k))-1$. Substituting this for $\chi(e)$ and dividing by $k$, the above inequality is equivalent to $2c_1(e)+3r(e) \ge (r(e)+1)k+3$, which follows easily from
$r(e) \ge 1$ and $c_1(e) > r(e)k$.

Thus, we may apply Theorem \ref{thm:O(k)-Segre-strata}
to deduce that the largest component of the $\oo(k)$-Segre stratum contains stable sheaves $E$ admitting short exact sequences \eqref{eq:ext-small-strata} for general stable $G$. Since $\chi(G)=\chi(e)-\chi(\oo(k))<0$,
such $E$ satisfy $h^0(E)=h^0(\oo(k))=\chi(\oo(k))$. Therefore, $B_{\chi(\oo(k))}(e)$ contains all such $E$ and thus has a component whose codimension is at most $1-\chi(e(-k))$. The last statement follows from the fact that $B_1(e)$ is by definition the intersection of $M^{\stable}(e)$ and the $\oo$-Segre stratum.
\end{proof}

\begin{rem}\label{rem:BN-expdim} The expected codimension of $B_j(e)$ in $M^{\stable}(e)$ is $j(j-\chi(e))$, see e.g. \cite[Thm. 2.3]{COSTA20101612}. Thus the corollary shows that $B_1(e)$ has the expected dimension. When $k>0$ and $\chi(e)<\chi(\oo(k))+1-r(e)-2(c_1(e)-r(e)k)/(k+3)$, the corollary shows that $B_{\chi(\oo(k))}(e)$ contains a locus of dimension larger than the expected dimension.
\end{rem}

We also obtain an analogue of Corollary \ref{cor:gen-res} for sheaves in non-general rank-one Segre strata. In these cases, the resolutions involve four line bundles rather than three.

\begin{cor}\label{cor:small-strata-res} Suppose $E \in M(e)$ is general in the
$\oo(k)$-Segre stratum, where $k_e<k<\mu(e)$. Let $g=e-[\oo(k)]$. Then $E$ admits a resolution
\[
  0 \to \oo(k_g-2)^{a_2} \to \oo(k_g-1)^{a_1} \oplus \oo(k_g)^{a_0} \oplus \oo(k) \to E \to 0  
\]
or
\[
    0 \to \oo(k_g-2)^{a_2} \oplus \oo(k_g-1)^{-a_1} \to \oo(k_g)^{a_0} \oplus \oo(k) \to E \to 0  
\]
depending on whether $a_1 \ge 0$ or $a_1 \le 0$. Here $a_0=\chi(g(-k_g))$, $a_1=\chi(g(-k_g+1))-3a_0$, and $a_2=a_0+a_1+1-r(e)$.
\end{cor}

\begin{proof}
$E$ admits a short exact sequence $0 \to \oo(k) \to E \to G \to 0$, with $G$ a general stable sheaf. Applying Corollary \ref{cor:gen-res} to $G$ gives the result.
\end{proof}

\bibliographystyle{alpha}
\bibliography{bib}
\end{document}